\RequirePackage{ifthen}
\newboolean{SIOPT}
\setboolean{SIOPT}{false}

\ifthenelse {\boolean{SIOPT}}
{
\documentclass[review]{siamart0516}
} {
\documentclass[10pt]{article}
\usepackage[margin=1.35in]{geometry}
}

\usepackage{graphicx}
\usepackage{amsmath}
\usepackage{amssymb}
\usepackage{hyperref}
\usepackage{color}

\usepackage{mathrsfs} % for \mathscr

\usepackage{soul} % for \ul

\usepackage{stmaryrd} % to use \Yup

\usepackage[makeroom]{cancel} % for \cancel

\newcommand{\R}{\mathbb R}
\newcommand{\Z}{\mathbb Z}

\let\st\relax
\DeclareMathOperator{\st}{s.t.}
\DeclareMathOperator{\conv}{conv}
\DeclareMathOperator{\poly}{poly}

\newcommand{\problems}{\mathscr P}
\newcommand{\solutions}{\mathscr S}

\newcommand{\ie}{i.e., }
\newcommand{\eg}{e.g.,~}

\newcommand{\C}{\mathcal C}
\newcommand{\D}{\mathcal D}

\renewcommand{\P}{\mathcal P}

\newcommand{\T}{\mathcal T}
\newcommand{\Qpoly}{\mathcal Q}

\newcommand{\ceil}[1]{\lceil#1\rceil}
\newcommand{\floor}[1]{\lfloor#1\rfloor}

\newcommand{\notemi}[1]{}

\newcommand{\abs}[1]{\left\lvert#1\right\rvert}
\newcommand{\pare}[1]{\left(#1\right)}

\usepackage{comment}% http://ctan.org/pkg/comment
\excludecomment{knownproof}
\ifthenelse {\boolean{SIOPT}}
{
\newenvironment{cpf}
{\begin{trivlist} \item[] {\hspace{.5cm} \em Proof of claim. }}
{$\hfill\diamond$ \end{trivlist} \medskip}

\newcounter{step}
\newenvironment{step}[1][]
{\refstepcounter{step} \smallskip \bfseries Step~\thestep. #1}
{. }

} {
\usepackage{amsthm}
\newtheorem{theorem}{Theorem}
\newtheorem{proposition}{Proposition}
\newtheorem{lemma}{Lemma}
\newtheorem{corollary}{Corollary}

\newenvironment{cpf}
{\begin{trivlist} \item[] {\em Proof of claim. }}
{$\hfill\diamond$ \end{trivlist} \medskip}

\newcounter{step}
\newenvironment{step}[1][]
{\refstepcounter{step} \begin{trivlist} \item[] \bfseries Step~\thestep. #1}
{\end{trivlist}}

}

\newtheorem{claim}{Claim}
\newtheorem{observation}{Observation}

\newcommand{\kywrds}{integer quadratic programming, approximation algorithm, concave function, subdeterminants, total unimodularity, total bimodularity}

\begin{document}

\title{Subdeterminants and Concave Integer Quadratic Programming\thanks{\textbf{Funding: }This work is 
%partially funded 
supported
by ONR grant N00014-19-1-2322. Any opinions, findings, and conclusions or recommendations expressed in this material are those of the authors and do not necessarily reflect the views of the Office of Naval Research.}}

\author{Alberto Del Pia
\thanks{Department of Industrial and Systems Engineering \& Wisconsin Institute for Discovery,
University of Wisconsin-Madison, Madison, WI, USA.
E-mail: {\tt delpia@wisc.edu}.}}

\date{\today}

\maketitle

\begin{abstract}
We consider the NP-hard problem of minimizing a separable concave quadratic function over the integral points in a polyhedron, and we denote by $\Delta$ the largest absolute value of the subdeterminants of the constraint matrix.
%This problem is NP-hard even if the constraint matrix is totally unimodular.
In this paper we give an algorithm that finds an $\epsilon$-approximate solution for this problem by solving a number of integer linear programs 
%of the same size of the original problem and 
whose constraint matrices have subdeterminants bounded by $\Delta$ in absolute value.
The number of these integer linear programs is polynomial in the dimension $n$, in $\Delta$ and in $1/\epsilon$, provided that the number $k$ of variables that appear nonlinearly in the objective is fixed.
%In particular, if $k$ is a fixed number and $\Delta \le 2$, we obtain the first polynomial-time approximation algorithm for this problem.
As a corollary, we obtain the first polynomial-time approximation algorithm for separable concave integer quadratic programming with $\Delta \le 2$ and $k$ fixed.
In the totally unimodular case $\Delta = 1$, we give an improved algorithm that only needs to solve a number of linear programs that is polynomial in $1/\epsilon$ and is independent of $n$, provided that $k$ is fixed.
%close the gap between the continuous and discrete version of the problem
%, and it is equivalent to its continuous relaxation obtained by dropping the integrality constraint.
%The continuous problem admits a strongly polynomial-time approximation algorithm, provided that the number of variables that appear nonlinearly in the objective is fixed.
%Our result in particular yields a strongly polynomial-time approximation algorithm for the integral minimum concave cost network flow problem with quadratic costs, provided that the number of nonlinear arc costs is fixed.
\end{abstract}

\ifthenelse {\boolean{SIOPT}}
{
% For SIOPT begin
\begin{keywords}
\kywrds
\end{keywords}

\begin{AMS}
90C10, 90C20, 90C26, 90C59
\end{AMS}
% For SIOPT end
}{
% For OO begin
\emph{Key words:} \kywrds
% For OO end
}

\section{Introduction}
\label{sec: intro}

In this paper we consider the problem of minimizing a separable concave quadratic function over the integral points in a polyhedron. % defined by a totally unimodular constraint matrix.
Formally, 
\begin{align}
\label{prob: IQP}
\tag{\ensuremath{\mathcal{IQP}}}
\begin{split}
\min &\quad \sum_{i=1}^k - q_i x_i^2 + h^\top x \\
% h^\top x - q^\top x^2 \\
\st & \quad Wx \le w \\
& \quad x \in \Z^n.
\end{split}
\end{align}
In this formulation, $x$ is the $n$-vector of unknowns and $k \le n$. 
%\notemi{We should have here $x_i \in \Z$, $i \in I$, instead of $x \in \Z^n$.}
%and $x^2$ denotes the vector $(x_1^2,\dots,x_n^2) \in \R^n$.
The matrix $W$ and the vectors $w,q,h$ stand for the data in the problem instance.
The vector $q$ is positive, and all the data is assumed to be integral:
%Without loss of generality, we assume that all the data is integral: 
$W \in \Z^{m \times n}$, $w \in \Z^m$, $q \in \Z^k_{> 0}$, and $h \in \Z^n$.
%Note that the integrality assumption on the vector $w$ is without loss of generality, since a more general $w \in \R^m$ can be replaced by  $\floor w$.
Problem \eqref{prob: IQP} is NP-hard even if $k=0$ as it reduces to integer linear programming.
%Moreover, it is known that any concave integer quadratic programming problem can be transformed in the form \eqref{prob: IQP} \cite{dP18}.
The concavity of the objective implies that \eqref{prob: IQP} can be solved in polynomial time for any fixed value of $n$ by enumerating the vertices of $\conv\{x \in \Z^n : Wx \le w\}$ \cite{Har89}.

A variety of important practical applications can be formulated with concave quadratic costs, including 
some aspects of VLSI chip design \cite{Wat84}, 
fixed charge problems \cite{Had64}, 
production and location problems \cite{Vai74}, 
bilinear programming \cite{Kon76a,VaiShe77},
and problems concerning economies of scale, which corresponds to the economic phenomenon of ``decreasing marginal cost'' \cite{Zwa74,RosPar86,FloVis95}.

%Throughout this paper we denote by $\Delta$ the largest absolute value of the subdeterminants of the constraint matrix $W$.

%\subsection{Our contribution}

In this paper we describe an algorithm that finds an $\epsilon$-approximate solution to 
%problem 
\eqref{prob: IQP} by solving a bounded number of integer linear programs (ILPs).
In order to state our approximation result, we first give the definition of $\epsilon$-approximation.
Consider an instance of a minimization problem that has an optimal solution, say $x^*$.
Let $f(x)$ denote the objective function
%assume that $f(x)$ is not constant over the feasible region,
%Let $x^*$ be an optimal solution of the problem, and 
and let $f_{\max}$ be the maximum value of $f(x)$ on the feasible region.
For $\epsilon \in [0,1]$, we say that a feasible point $x^\diamond$ is an \emph{$\epsilon$-approximate solution} if $f(x^*) = f_{\max}$, or if $f(x^*) < f_{\max}$ and
\begin{equation}
\label{eq: epsilon}
%f(x^\diamond) - f(x^*) \le \epsilon \cdot (f_{\max} - f(x^*)).
\frac{f(x^\diamond) - f(x^*)}{f_{\max} - f(x^*)} \le \epsilon.
\end{equation}
Note that if $f(x^*) = f_{\max}$, then we have $f(x^\diamond) = f(x^*)$ and $x^\diamond$ is an optimal solution.
In the case that the problem is infeasible or unbounded, an $\epsilon$-approximate solution is not defined, and we expect our algorithm to return an indicator that the problem is infeasible or unbounded.
%If $f(x)$ is constant over the feasible region, then we expect our algorithm to return a feasible solution which is therefore optimal.
If the objective function has no upper bound on the feasible region, our definition loses its value because any feasible point is an $\epsilon$-approximation for any $\epsilon > 0$.
The definition of $\epsilon$-approximation has some useful invariance properties which make it a natural choice for unstructured problems.
For instance, it is preserved under dilation and translation of the objective function, and it is insensitive to affine transformations of the objective function and of the feasible region.
Our definition of approximation has been used in earlier works, and we refer to  \cite{NemYud83,Vav92c,BelRog95,KleLauPar06} for more details.

The running time of our algorithm depends on the largest absolute value of any subdeterminant of the constraint matrix $W$ in \eqref{prob: IQP}.
As is customary, throughout this paper, we denote this value by $\Delta$.
While there has been a stream of recent studies that link $\Delta$ to the complexity of ILP (see, e.g., \cite{ArtEisGlaOerVemWei16,ArtWeiZen17,PaaSchWei19}), only few papers explored how $\Delta$ affects nonlinear problems (see Section~\ref{sec: related}).
The following is our main result.

\begin{theorem}
\label{th: delta}
%Consider a problem \eqref{prob: IQP}, and denote 
%by $k$ the number of variables that appear nonlinearly in the objective function, and 
%by $\Delta$ the largest absolute value of the subdeterminants of the constraint matrix $W$.
For every $\epsilon \in (0,1]$ there is an algorithm that finds an $\epsilon$-approximate solution to
%problem
\eqref{prob: IQP}
by solving
\begin{align*}
\pare{3 + \left\lceil\sqrt{k \left((2n\Delta)^2 + \frac{1}{\epsilon}\right)}\right\rceil }^k
\end{align*}
ILPs of size polynomial in the size of 
%problem 
\eqref{prob: IQP}. 
Moreover, each ILP has integral data,
at most $n$ variables, 
at most $m$ linear inequalities and possibly additional variable bounds,
and a constraint matrix with subdeterminants bounded by $\Delta$ in absolute value.
\end{theorem}
Assume now that $k$ is a fixed number.
In this case, the number of ILPs that our algorithm solves is polynomial in $n$, $\Delta$, and $1/\epsilon$.
Hence, Theorem~\ref{th: delta} implies that the discovery of a polynomial-time algorithm for ILPs with subdeterminants bounded by some polynomial in the input size, would directly imply the existence of a polynomial-time approximation algorithm for \eqref{prob: IQP} with subdeterminants bounded by the same polynomial.

Consider now the case $\Delta \le 2$.
Then Theorem~\ref{th: delta} and the strongly polynomial-time solvability of totally bimodular ILPs \cite{ArtWeiZen17} imply the following result.
\begin{corollary}
\label{cor: delta}
Consider problem \eqref{prob: IQP} where $\Delta \le 2$.
%Consider a problem \eqref{prob: IQP} where the subdeterminants of the constraint matrix $W$ are bounded by $2$ in absolute value.
For every $\epsilon \in (0,1]$ there is an algorithm that finds an $\epsilon$-approximate solution 
%to problem \eqref{prob: IQP} 
%to \eqref{prob: IQP} with $\Delta \le 2$ 
in a number of operations bounded by
\begin{align*}
\pare{3 + \left\lceil\sqrt{k \left((4n)^2 + \frac{1}{\epsilon}\right)}\right\rceil}^k \poly(n,m).
\end{align*}
%In the above formula, $k$ denotes the number of variables that appear nonlinearly in the objective function.
\end{corollary}

If $k$ is fixed, Corollary~\ref{cor: delta} implies that we can find an $\epsilon$-approximate solution in a number of operations that is polynomial in the number of variables and constraints of the problem and in $1/\epsilon$.
In particular, the number of operations is strongly polynomial in the input size since it is independent of the vectors $q,h,w$ in \eqref{prob: IQP}.
We remark that this is the first known polynomial-time approximation algorithm for this problem.

When $\Delta \le 2$, our result closes the gap between the best known algorithms for 
%the discrete problem 
\eqref{prob: IQP} and its continuous version
%Consider now the concave quadratic programming problem 
obtained %from  \eqref{prob: IQP} 
by dropping the integrality constraint:
\begin{align}
\label{prob: CQP}
\tag{\ensuremath{\mathcal{CQP}}}
\begin{split}
\min & \quad \sum_{i=1}^k -q_i x_i^2 + h^\top x \\
% h^\top x - q^\top x^2 \\
\st & \quad Wx \le w \\
& \quad x \in \R^n.
\end{split}
\end{align}
In fact, in \cite{Vav92c} Vavasis gives an algorithm that finds an $\epsilon$-approximate solution to \eqref{prob: CQP} in strongly polynomial-time when $k$ is fixed.

%We now explain why an approximation algorithm for \eqref{prob: CQP} does not necessarily yield an approximation algorithm for \eqref{prob: IQP}.
%The most apparent reason is that an $\epsilon$-approximation $x^\diamond$ to \eqref{prob: CQP} is not necessarily integral, and therefore is generally not feasible for \eqref{prob: IQP}.
%Even if %we could find 
%there existed a feasible solution for \eqref{prob: IQP} with objective value lower than that of $x^\diamond$, such vector would not necessarily be an $\epsilon$-approximation to \eqref{prob: IQP}, since $f_{\max} - f(x^*)$ in the definition of $\epsilon$-approximation can be much larger for problem \eqref{prob: CQP} than for problem \eqref{prob: IQP}.
%%First, an $\epsilon$-approximation $x^\diamond$ to \eqref{prob: CQP} is not necessarily integral, and therefore is generally not feasible for \eqref{prob: IQP}.
%%Second, the $f_{\max}$ in the definition of $\epsilon$-approximation can be much larger for problem \eqref{prob: CQP} than for problem \eqref{prob: IQP}.
%In fact, it is simple to to construct a family of instances where $f(x^*)$ is constant, the discrete $f_{\max}$ is constant and the continuous $f_{\max}$ can be arbitrarily large.

% by giving an approximation algorithm for problem \eqref{prob: IQP} that runs in strongly polynomial time provided that $k$ is a fixed number.

\subsection{The totally unimodular case}

A fundamental special case of 
%our problem 
\eqref{prob: IQP} is when $W$ is totally unimodular (TU), \ie $\Delta = 1$.
Examples of TU matrices include incidence matrices of directed graphs and of bipartite graphs, matrices with the consecutive-ones property, and network matrices (see, \eg \cite{SchBookIP}). 
A characterization of TU matrices is given by Seymour~\cite{Sey80}.
Many types of applications can be formulated with a TU constraint matrix, including a variety of network and scheduling problems.

% we can assume without loss of generality that $w$ is an integral vector.
If the matrix $W$ is TU, a fundamental result by Hoffman and Kruskal \cite{HofKru56} implies that the polyhedron defined by $Wx \le w$ is integral.
Together with the concavity of the objective function this implies that it is polynomially equivalent to solve \eqref{prob: IQP} and \eqref{prob: CQP} to global optimality.
%Hence, both 
%%problems 
%\eqref{prob: IQP} and \eqref{prob: CQP} are NP-hard even if $W$ is TU.
Problem \eqref{prob: CQP} contains as a special case the minimum concave-cost network flow problem with quadratic costs, which is NP-hard as shown by a reduction from the subset sum problem  (see proof in \cite{GuiPar90} for strictly concave costs).
Therefore, both 
%problems 
\eqref{prob: CQP} and \eqref{prob: IQP} are NP-hard even if $W$ is TU.
%The complexity of problems \eqref{prob: IQP} and \eqref{prob: CQP} is unknown if $W$ is TU and $k$ is fixed.
%Pardalos and Vavasis \cite{ParVav91} showed that minimizing a separable concave quadratic function over a polyhedron is NP-hard even if only one variable appears nonlinearly in the objective.
%However, the existence of strongly polynomial-time algorithms for linear programming and integer linear programming with TU constraint matrices \cite{Tar86,HofKru56} suggests that the special combinatorial structure of the constraints may significantly reduce the complexity of problems \eqref{prob: IQP} and \eqref{prob: CQP}.

In the special case where $W$ is TU 
%we give an improved algorithm which only needs to solve 
we give an approximation algorithm which improves on the one of Theorem~\ref{th: delta} since it only needs to solve a number of linear programs (LPs) that is independent of the dimension $n$.

\begin{theorem}
\label{th: TU}
Consider problem \eqref{prob: IQP} where 
%the constraint matrix 
$W$ is TU.
For every $\epsilon \in (0,1]$ there is an algorithm that finds an $\epsilon$-approximate solution 
%to problem \eqref{prob: IQP} 
by solving
\begin{align*}
\pare{3 + \left\lceil\sqrt{k \left(1 + \frac{1}{\epsilon}\right)}\right\rceil }^k
\end{align*} 
LPs of size polynomial in the size of 
%problem 
\eqref{prob: IQP}. 
Moreover, each LP has % integral data,
at most $n$ variables, 
at most $m$ linear inequalities and possibly additional variable bounds,
and a TU constraint matrix.
\end{theorem}

%Note that in Theorem~\ref{th: TU} we solve LPs instead of ILPs, since they are equivalent corresponding ILPs in Theorem~\ref{th: delta}, since 
%In the special case where $W$ is TU, the ILPs in  can be solved via linear programming.

%In this special setting we give an approximation algorithm which improves on the one of Theorem~\ref{th: delta} since it only needs to solve 
Since each LP with a TU constraint matrix can be solved in strongly polynomial time \cite{Tar86},
% manage to remove the dependance on $n$ in the number of ILPs to solve, and 
Theorem~\ref{th: TU} implies the following result.
\begin{corollary}
\label{cor: TU}
Consider problem \eqref{prob: IQP} where 
%the constraint matrix 
$W$ is TU.
For every $\epsilon \in (0,1]$ there is an algorithm that finds an $\epsilon$-approximate solution 
%to problem \eqref{prob: IQP} 
in a number of operations bounded by
\begin{align*}
\pare{3 + \left\lceil\sqrt{k \left(1 + \frac{1}{\epsilon}\right)}\right\rceil }^k \poly(n,m).
\end{align*}
%In the above formula, $k$ denotes the number of variables that appear nonlinearly in the objective function.
\end{corollary}

\section{Related problems and algorithms}
\label{sec: related}

In this section we present optimization problems and algorithms that are closely related to the ones presented in this paper and we discuss their connection with our result. 
In Section~\ref{sec: literature} we review a number of related optimization problems and the state-of-the-art regarding their complexity.
In Section~\ref{sec: proof techniques} we discuss mesh partition and linear underestimators, which is the classic technique our algorithm builds on. %This is the  based on  have been successfully applied to other optimization problems in the past decades.
Finally, in Section~\ref{sec: extensions} we discuss potential extensions and open questions.

\subsection{Related problems}
\label{sec: literature}

%To the best of our knowledge, problem \eqref{prob: IQP} has not yet been studied in the proposed form. 
To the best of our knowledge, problem \eqref{prob: IQP} has not yet been studied in this generality.
In this section we present the state-of-the-art regarding optimization problems that are closely related to \eqref{prob: IQP}.

\subsubsection{Separable problems}

Some exact algorithms are known for the problem of minimizing a separable function over the integral points in a polytope. % defined by a TU constraint matrix.

Horst and Van Thoai \cite{HorVan96} give a branch and bound algorithm for the case where the objective function is separable concave, the constraint matrix is TU, and box constraints $0 \le x \le u$ are explicitly given.
They obtain an algorithm that performs a number of operations that is polynomial in $m,n$ and the maximum $u_i$ among the bounds on the nonlinear variables, provided that the number of variables that appear nonlinearly in the objective is fixed.
In the worst case, this algorithm performs a number of operations that is exponential in the size of the vector $u$. %, and therefore it is not faster than 
% asymptotically with the running time of the algorithm that enumerates all possible nonlinear variables in the box $[0,u]$ and solves, for each, the restricted problem, which is an ILP with a TU constraint matrix.
An algorithm that carries out a comparable number of operations can be obtained by enumerating all possible subvectors of nonlinear variables in the box $[0,u]$ and solving, for each, the restricted problem, which is an ILP with a TU constraint matrix. % that can be solved in polynomial time.
To the best of our knowledge this is currently the best known algorithm to solve 
%problem 
\eqref{prob: IQP} with a TU constraint matrix.

Meyer \cite{Mey77} gives a polynomial-time algorithm for the case where the objective function is separable convex, the feasible region is bounded, and the constraint matrix is TU.
Hochbaum and Shanthikumar \cite{HocSha90} extend this result by giving a polynomial-time algorithm for the case where the objective function is separable convex, the feasible region is bounded, and the largest subdeterminant of the constraint matrix is polynomially bounded.
%See \cite{Onn08} for extensions to more general convex objective functions.

\subsubsection{Polynomial problems}

A number of algorithms are known for the problem of optimizing a polynomial function over the mixed-integer points in a polyhedron. 

De Loera et al.~\cite{DeLHemKopWei08} present an algorithm to find an $\epsilon$-approximate solution to the problem of minimizing a polynomial function over the mixed-integer points in a polytope.
The number of operations performed is polynomial in the maximum total degree of the objective, the input size, and $1/\epsilon$, provided that the dimension is fixed.
They also give a fully polynomial-time approximation scheme for the problem of maximizing a nonnegative polynomial over mixed-integer points in a polytope, when the number of variables is fixed.

Del Pia et al.~\cite{dPDeyMol17} give a pseudo polynomial-time algorithm for the problem of minimizing a quadratic function over the mixed-integer points in a polyhedron when the dimension is fixed.

Hildebrand et al.~\cite{HilWeiZem16} give a fully polynomial-time approximation scheme for the problem of minimizing a quadratic function over the integral points in a polyhedron, provided that the dimension is fixed and the objective is homogeneous with at most one positive or negative eigenvalue.

Del Pia~\cite{dP16,dP18} gives an algorithm that finds an $\epsilon$-approximate solution to the problem of minimizing a concave quadratic function over the mixed-integer points in a polyhedron.
The number of operations is polynomial in the input size and in $1/\epsilon$, provided that the number of integer variables and the number of negative eigenvalues of the objective function are fixed.

Note that all these algorithms carry out a polynomial number of operations only if the number of integer variables is fixed.
This is in contrast with the results presented in this paper. 
Our assumptions on the separability of the objective and on the subdeterminants of the constraint matrix allow us to 
consider a general (not fixed) number of integer variables.
%obtain a polynomial running time also if the number of integer variables is not fixed.

\subsubsection{Minimum concave cost network flow problem}

One of the most challenging problems of network optimization is the \emph{minimum concave cost network flow problem} (MCCNFP). 
Given a digraph $(V,A)$, the MCCNFP is defined as
\begin{align*}
\min & \quad \sum_{a \in A} c_a (x_a) \\
\st & \quad \sum_{a \in \delta^+(v)} x_a - \sum_{a \in \delta^-(v)} x_a = b(v) && \forall v \in V \\
& \quad 0 \le x_a \le u_a && \forall a \in A,
\end{align*}
where $c_a$ is the cost function for arc $a$, which is nonnegative and concave, $b(v)$ is the supply at node $v$, $\delta^+(v)$ and $\delta^-(v)$ are the set of outgoing and incoming arcs at node $v$, respectively, and $u_a$ is a bound on the flow $x_a$ on arc $a$.
For a discussion on the applications and a review of the literature on this problem, we refer the reader to the articles of Guisewite and Pardalos \cite{GuiPar90,GuiPar91}.
The MCCNFP is closely related to \eqref{prob: CQP} with a TU constraint matrix, since its constraint matrix is TU 
and its objective is separable and concave.
As we already mentioned, MCCNFP is NP-hard even with quadratic costs \cite{GuiPar90}, and its complexity is unknown if we assume that the number of nonlinear arc costs is fixed.
In view of its relevance to numerous applications, %in operations research, economics, engineering, etc. 
the MCCNFP has been the subject of intensive research.
Tuy et al.~\cite{TuyGhaMigVar95} give a polynomial-time algorithm for MCCNFP provided that the number of sources and nonlinear arc costs is fixed. % (see also \cite{GuiPar93}).
See \cite{HeAhmNem15} and references in \cite{TuyGhaMigVar95} for other polynomially-solvable cases of the MCCNFP.

Like for general network flow problems, it is natural to consider the discrete version of the MCCNFP problem, where all flows on the arcs are required to be integral.
Our results in particular yield an algorithm to find an $\epsilon$-approximate solution to the integral MCCNFP with quadratic costs.
%Our algorithm reduces the problem to solving a number of LP problems that is a polynomial function of $1/\epsilon$, provided that the number of nonlinear arc costs is fixed.
The number of operations performed by this algorithm is polynomial in the size of the digraph ($|V|$ and $|A|$) and in $1/\epsilon$, provided that the number of nonlinear arc costs is fixed.
In particular, the number of operations is independent of the quadratic costs, the supply vector, and the flow bounds.

\subsection{Proof techniques}
\label{sec: proof techniques}

Our algorithms build on the classic technique
%based on 
of
mesh partition and linear underestimators.
This natural approach consists of replacing the nonlinear objective function by a piecewise linear approximation, an idea known in the field of optimization since at least the 1950s.
This general algorithmic framework is used in a variety of contexts in science and engineering, and the literature on them is expansive (see, \eg \cite{ParRos87,NemWol88,HorVan96,MagStr12}).

In the early 1990s Vavasis designed approximation algorithms for quadratic programming based on mesh partition and linear underestimators \cite{Vav92c,Vav92i}.
%these techniques.
His most general result is 
%He gave 
a polynomial-time algorithm to find an $\epsilon$-approximate solution for the case where the objective has a fixed number of negative eigenvalues.
%, and he obtained a weaker polynomial-time approximation algorithm for the general indefinite case \cite{Vav93}.
One of the main difficulties in proving these results consists in giving a lower bound on the value $f_{\max} - f(x^*)$ in the definition of $\epsilon$-approximate solution.
Vavasis' idea consists in constructing two feasible points along the most concave direction of the objective function, and then using their midpoint to obtain the desired bound.

In \cite{dP16,dP18} Del Pia employs mesh partitions and linear underestimators in concave mixed-integer quadratic programming.
He gives an algorithm that finds an $\epsilon$-approximate solution in polynomial-time, provided that
% for the case where 
%the objective is concave and 
the number of negative eigenvalues and the number of integer variables are fixed.
Vavasis' technique is not directly applicable to the mixed-integer setting since the midpoint of two feasible points is generally not feasible.
To obtain the desired bound, these algorithms decompose the original problem into a fixed number of subproblems.
The geometry of the mixed-integer points guarantees that in each subproblem the midpoint is feasible and this is used to obtain the desired bound.

The decomposition approaches introduced in \cite{dP16,dP18} are not effective if the number of integer variables is not fixed.
The flatness-based algorithm described in \cite{dP16} could yield an exponential number of subproblems, and their constraint matrices can have subdeterminants larger than those of the original constraint matrix. % not generally defined by a TU constraint matrix.
The parity-based algorithm introduced in \cite{dP18} would not increase the subdeterminants, but it would yield 
%a number of subproblems exponential in $n$.
$2^n$ subproblems.
%subproblems defined by TU matrices but their number would still be exponential in $n$.
%One could think of applying the parity-based decomposition only to the nonlinear variables, but this 
%%Moreover, these decompositions do not maintain the polyhedral structure of the feasible region and in the setting proposed in this paper they 
%would yield subproblems over which it is NP-hard to optimize even a linear function, even in the TU case~\cite{ConDiSEisWol09}.
%To overcome these difficulties, in this paper we introduce a novel decomposition technique. 
%On the one hand this technique does not increase the subdeterminants in the subproblems.
%%yields subproblems that have the same structure of the original problem.
%%\note{decompsition tailored to have property and also maintain TU or whatever}
%On the other hand, while in each subproblem we cannot guarantee that the midpoint used to obtained the bound is feasible, the special combinatorial structure of the constraints allows us to find a feasible point, in a suitable neighborhood of the midpoint, with objective value close enough to that of the midpoint.
%%This feasible point can be found by solving an ILP in a suitable neighborhood of the midpoint.
%The obtained bound on the objective value of the feasible point allows us to give the desired bound on the value $f_{\max} - f(x^*)$.
To overcome these difficulties, in this paper we introduce a novel decomposition technique which does not increase the subdeterminants in the subproblems, and that generates a number of subproblems that is polynomial in $n,\Delta,\frac 1 \epsilon$, provided that $k$ is a fixed value.
%yields subproblems that have the same structure of the original problem.
%\note{decompsition tailored to have property and also maintain TU or whatever}
While in each subproblem we cannot guarantee that the midpoint used to obtained the bound is feasible, the special combinatorial structure of the constraints allows us to show the existence of a feasible point with objective value close enough to that of the midpoint.
%, in a suitable neighborhood of the midpoint.
%This feasible point can be found by solving an ILP in a suitable neighborhood of the midpoint.
The obtained bound on the objective value of this feasible point allows us to give the desired bound on the value $f_{\max} - f(x^*)$.

\subsection{Extensions and open questions}
\label{sec: extensions}

%\notemi{In Step 3 it is written how to do that. Check first that everything works out. The only drawback is that the running time for the mixed-integer case will be worse.}

The algorithms presented in this paper can also be applied to problems with any objective function sandwiched between two separable concave quadratic functions.
This is a consequence of a property of $\epsilon$-approximate solutions that we now present.

Consider an instance $I$ of a minimization problem that has an optimal solution, say $x^*$.
Let $f(x)$ denote the objective function,
% and by $f_{\min}$ and 
and let $f_{\max}$ be the maximum value of $f(x)$ on the feasible region.
Let $f'(x)$ be a function such that for every feasible $x$ we have
\begin{align}
\label{eq: sandwich}
f(x) \le f'(x) \le f(x) + \xi (f_{\max} - f(x^*)),
\end{align}
where $\xi$ is a parameter in $[0,1)$.
Denote by $I'$ the instance obtained from $I$ by replacing the objective function $f(x)$ with $f'(x)$.

\begin{observation}
For every $\epsilon' \in (\frac{\xi}{1- \xi},1]$, any $\epsilon$-approximate solution to $I$, where $\epsilon := \epsilon' (1- \xi) - \xi$, is an $\epsilon'$-approximate solution to $I'$.
%There is an algorithm to find an $\epsilon'$-approximate solution to $I'$ for every $\epsilon' \in (\frac{\xi}{1- \xi},1]$.
%The running time is the one given in Theorem \ref{th: delta} with $\epsilon := \epsilon' (1- \xi) - \xi$.
\end{observation}

%\note{Check it works if $f(x^*) = f_{\max}$}

\begin{proof}
Let $\epsilon := \epsilon' (1- \xi) - \xi$ and note that $\epsilon \in (0,1]$.
Let $x^\diamond$ be an $\epsilon$-approximate solution to $I$. % that can be found by Theorem \ref{th: delta}.
We show that $x^\diamond$ is an $\epsilon'$-approximate solution to $I'$.
%If $f(x^*) = f_{\max}$, then $x^\diamond$ is an optimal solution to both $I$ and $I'$, thus we assume $f(x^*) < f_{\max}$.

Let ${x^*}'$ be an optimal solution to $I'$, and let $f'_{\max}$ be the maximum value of $f'(x)$ on the feasible region.
If $f(x^*) = f_{\max}$, then from \eqref{eq: sandwich} we have $f({x^*}') = f'_{\max}$, and $x^\diamond$ is an $\epsilon'$-approximate solution to $I'$.
Therefore, in the remainder of the proof we assume $f(x^*) < f_{\max}$.
Using the inequalities \eqref{eq: sandwich} we obtain 
$f(x^*) \le f'({x^*}') \le f(x^*) + \xi (f_{\max} - f(x^*))$ and 
$f_{\max} \le f'_{\max} \le f_{\max} + \xi (f_{\max} - f(x^*))$.
Hence
\begin{align*}
\frac{f'(x^\diamond) - f'({x^*}')}{f'_{\max} - f'({x^*}')} 
& \le
\frac{f(x^\diamond) - f(x^*) + \xi (f_{\max} - f(x^*))}{f_{\max} - f(x^*) - \xi (f_{\max} - f(x^*))} \\
& \le
\frac{\epsilon(f_{\max} - f(x^*)) + \xi (f_{\max} - f(x^*))}{f_{\max} - f(x^*) - \xi (f_{\max} - f(x^*))} \\
& = 
\frac{\epsilon + \xi}{1- \xi}
=
\frac{\epsilon' (1- \xi) - \xi + \xi}{1- \xi} 
=
\epsilon'.
\end{align*}
This shows that $x^\diamond$ is an $\epsilon'$-approximate solution to $I'$.
\end{proof}

We conclude this section by posing some natural open questions.
What is the computational complexity of problems \eqref{prob: IQP} and \eqref{prob: CQP}, if we assume that $k$ is fixed and that
%the number $k$ of variables that appear nonlinearly in the objective is fixed?
the subdeterminants of $W$ are bounded by either $1$ or $2$ in absolute value?
%$\Delta$ is either $1$ or $2$, and that $k$ is fixed?
Does there exist a polynomial-time algorithm that solves them exactly, or are they NP-hard?
%These questions are open even if we restrict ourselves to feasible regions of the form of MCCNFP.
%What is the computational complexity of problems MCCNFP and integral MCCNFP with quadratic costs, if we assume that the number $k$ of nonlinear arc costs is fixed?
To the best of our knowledge, all these questions are open even 
if we restrict ourselves to the case $k=1$, or to feasible regions of the form of MCCNFP.

Another interesting open question regards the problem obtained from \eqref{prob: IQP} by considering a general separable quadratic objective function.
In this setting, each variable that appears nonlinearly in the objective has a cost function that is either a convex or concave quadratic.
Does there exists a polynomial-time algorithm that finds an $\epsilon$-approximate solution to this problem, if we assume that $W$ is TU and that the number of concave variables is fixed?
The algorithm presented in this paper does not seem to extend to this case, even if we make the stronger assumption that the total number of variables that appear nonlinearly in the objective is fixed.
The main reason is that, in this setting, we are not able to give a suitable lower bound on the value $f_{\max} - f(x^*)$ in the definition of $\epsilon$-approximate solution.
This is because the two feasible points constructed along the most concave direction of the objective function might not be aligned in the convex directions, thus not even their midpoint yields the desired bound.
%Simultaneous diagonalization techniques used by Vavasis for indefinite quadratic programming \cite{Vav92i} do not seem to be useful to fix this shortcoming since they would yield subproblems not generally defined by TU constraint matrices.

% =========================================================

\section{Approximation algorithm}

\subsection{Description of the algorithm}
\label{sec: algorithm description}

%\note{Write that it is a classic algorithm. Highlight the two key differences.}

In this section we describe our algorithm to find an $\epsilon$-approximate solution to 
%problem 
\eqref{prob: IQP}.
The main difference from a standard algorithm based on mesh partition and linear underestimators is the decomposition of the problem in Step~2, and the specific choice of the mesh in Step~3.

Consider now our input problem \eqref{prob: IQP}, and recall that $\Delta$ denotes the largest absolute value of any subdeterminant of the constraint matrix $W$.
We also assume that $k \ge 1$, as otherwise the problem is an ILP.

\begin{step}
\label{step: prelim}
Feasibility and boundedness\end{step}
For every $i=1,\dots,k$, solve the two ILPs
\begin{align}
\label{pr: bounds}
\begin{split}
\min & \{x_i : Wx \le w, \ x \in \Z^n \}, \\
\max & \{x_i : Wx \le w, \ x \in \Z^n \}.
\end{split}
\end{align}
If any of these ILPs are infeasible, then the algorithm returns that 
%problem 
\eqref{prob: IQP} is infeasible. % and terminates.
If any of the ILPs in \eqref{pr: bounds} are unbounded, then the algorithm returns that 
%problem 
\eqref{prob: IQP} is unbounded. % and terminates.
Otherwise, let $\bar x$ be an integral vector that satisfies $Wx \le w$, which can be, for example, an optimal solution of one of the $2k$ ILPs just solved.

Solve the ILP
\begin{align}
\label{pr: aux}
\min \{ h^\top x : Wx \le w, \ x_i = \bar x_i, i=1,\dots,k, \ x \in \Z^n \}.
\end{align}
%If problem \eqref{pr: aux} is infeasible (this can only happen for $k=0$), then the algorithm returns that problem \eqref{prob: IQP} is infeasible and terminates.
If %problem 
\eqref{pr: aux} is unbounded, then the algorithm returns that 
%problem 
\eqref{prob: IQP} is unbounded. % and terminates.
Otherwise,
%problem 
\eqref{prob: IQP} is feasible and bounded.

Initialize the list of problems to be solved as $\problems := \{$\eqref{prob: IQP}$\}$, and the list of possible approximate solutions to \eqref{prob: IQP} as $\solutions := \emptyset$.

\begin{step}
\label{step: dec}
Decomposition\end{step}

If $\problems = \emptyset$, then the algorithm returns the solution in $\solutions$ with the minimum objective function value.
Otherwise $\problems \neq \emptyset$, and let \eqref{prob: IQPtilde} be a problem in $\problems$.
%Let \eqref{prob: IQPtilde} be the chosen problem.
%If $\problems = \emptyset$, return the solution in $\solutions$ with the minimum objective function value.
%Select a problem \eqref{prob: IQPtilde} in $\problems$. 

Clearly, in the first iteration we have \eqref{prob: IQPtilde} = \eqref{prob: IQP}.
It will be clear from the description of the algorithm that, at a general iteration, \eqref{prob: IQPtilde} is obtained from 
%problem 
\eqref{prob: IQP} by 
%possibly 
fixing 
%and then dropping 
a number of variables $x_i$, $i=1,\dots,k$, to integer values.
Thus, 
%by eventually dropping and reordering components of vectors $x,q,c,h$, and by dropping a constant term in the objective, \eqref{prob: IQPtilde} is a bounded problem of the form
by eventually dropping a constant term in the objective, \eqref{prob: IQPtilde} is a bounded problem of the form
\begin{align}
\label{prob: IQPtilde}
\tag{\ensuremath{\widetilde{\mathcal{IQP}}}}
\begin{split}
\min &\quad \sum_{i=1}^{\tilde k} - \tilde q_i x_i^2 + \tilde h^\top x \\
\st & \quad \widetilde W x \le \tilde w \\
& \quad x \in \Z^{\tilde n}.
\end{split}
\end{align}
In this formulation, $x$ is the $\tilde n$-vector of unknowns and we have $\tilde n = \tilde k + n - k$.
The constraint matrix $\widetilde W \in \Z^{m \times \tilde n}$ 
%has subdeterminants bounded in absolute value by $\Delta$, 
is a column submatrix of $W$,
$\tilde w \in \Z^m$, 
$\tilde q \in \Z^{\tilde k}_{>0}$ is a subvector of $q$, 
and $\tilde h \in \Z^{\tilde n}$ is a subvector of $h$.
We remark that the variables $x_1,\dots,x_{\tilde k}$ in the formulation of \eqref{prob: IQPtilde} are not necessarily the first $\tilde k$ variables as ordered in \eqref{prob: IQP}, but rather a subset of $\tilde k$ variables of the original $k$ variables $x_1,\dots,x_k$.

If $\tilde k=0$, find an optimal solution to \eqref{prob: IQPtilde}, which is an ILP.
Add the corresponding solution to 
%problem 
\eqref{prob: IQP} (obtained by restoring the $n - \tilde n$ components of $x \in \R^n$ fixed to obtain \eqref{prob: IQPtilde} from \eqref{prob: IQP}) to $\solutions$, remove \eqref{prob: IQPtilde} from $\problems$, and go back to Step~\ref{step: dec}.
Otherwise, 
for every $i=1,\dots,\tilde k$, solve the two bounded ILPs
\begin{align}
\label{eq: boundstilde}
\begin{split}
\tilde l_i & := \min \{x_i : \widetilde Wx \le \tilde w, \ x \in \Z^{\tilde n} \}, \\
\tilde u_i & := \max \{x_i : \widetilde Wx \le \tilde w, \ x \in \Z^{\tilde n} \}.
\end{split}
\end{align}
If any of these ILPs are infeasible, then 
%the algorithm removes 
remove \eqref{prob: IQPtilde} from $\problems$ and go back to Step~\ref{step: dec}.

%\notemi{The decomposition below should be done only on integer variables. Do we have $\tilde p \Delta$ instead of $\tilde n \Delta$?}

Let $\tilde g := \left\lceil\sqrt{\tilde k \left((2 \tilde n \Delta)^2 + 1/\epsilon \right)}\right\rceil$.
If there exists an index $i \in \{1,\dots,\tilde k\}$ such that $\tilde u_i-\tilde l_i < \tilde g$, replace 
%problem 
\eqref{prob: IQPtilde} in $\problems$ with all the subproblems of \eqref{prob: IQPtilde} obtained by fixing the variable $x_i$ to each integer value between $\tilde l_i$ and $\tilde u_i$, and go back to Step~\ref{step: dec}.
%Otherwise, 
If there is no index $i \in \{1,\dots,\tilde k\}$ such that $\tilde u_i-\tilde l_i < \tilde g$,
%if $\tilde u_i-\tilde l_i \ge \tilde g$ for every index $i \in \{1,\dots,\tilde k\}$, 
continue with Step~\ref{step: mesh}.

\begin{step}
\label{step: mesh}
Mesh partition and linear underestimators\end{step}
Let $\Qpoly \subset \R^{\tilde k}$ be the polytope defined by
\begin{align*}
\Qpoly := \{ (x_1,\dots,x_{\tilde k}) \in \R^{\tilde k} : \tilde l_i \le x_i \le \tilde u_i, \ i = 1,\dots,\tilde k\}.
\end{align*}
Place a $(\tilde g+1) \times \cdots \times (\tilde g+1)$ grid of points in $\Qpoly$ defined by
\begin{align*}
\left\{
\begin{pmatrix}
\tilde l_1 \\
\tilde l_2 \\
\vdots \\
\tilde l_{\tilde k}
\end{pmatrix}
+ 
\frac{1}{\tilde g}
\begin{pmatrix}
i_1 (\tilde u_1 - \tilde l_1) \\
i_2 (\tilde u_2 - \tilde l_2) \\
\vdots \\
i_{\tilde k} (\tilde u_{\tilde k} - \tilde l_{\tilde k})
\end{pmatrix}
: 
i_1,\dots,i_{\tilde k} \in \{0, 1,\dots, \tilde g\}
\right\}.
\end{align*}
The grid partitions $\Qpoly$ into $\tilde g^{\tilde k}$ boxes.

For each box $\C = [r_1, s_1] \times \cdots \times [r_{\tilde k}, s_{\tilde k}] \subset \R^{\tilde k}$, among the $\tilde g^{\tilde k}$ boxes just constructed, define the affine function $\mu: \R^{\tilde n} \to \R$ 
as %defined by
%that attains the same values as $q(x)$ at the vectors corresponding to the vertices of $\D$:
%that attains the same values as $\sum_{i=1}^{\tilde k} -q_i x_i^2$ at the vectors corresponding to the vertices of $\C$:
\begin{align}
\label{eq: underestimator}
\mu(x) := \sum_{i=1}^{\tilde k} (-\tilde q_i(r_i+s_i) x_i + \tilde q_i r_i s_i),
\end{align}
and solve the bounded ILP
\begin{align}
\label{prob: ILP on box}
\begin{split}
\min & \quad \mu(x)+ \tilde h^\top x \\
\st & \quad \widetilde Wx \le \tilde w \\
& \quad \ceil{r_i} \le x_i \le \floor{s_i} \qquad i=1,\dots,\tilde k \\
& \quad x \in \Z^{\tilde n}.
\end{split}
\end{align}
%\notemi{We still need to cover all feasible points, but we need integer bounds to solve these problems. I think that we should have bounds $\floor{r_i} \le x_i \le \ceil{s_i}$ instead for the continuous variables. For integer variables we can keep $\ceil{r_i} \le x_i \le \floor{s_i}$. We will then obtain an integral solution which is then feasible. The boxes defined in this way will partially overlap. This will probably change all the resulting bounds in the claims.}

Let $x^\diamond$ be the best solution among all the (at most) $\tilde g^{\tilde k}$ optimal solutions just obtained.
%, and let $x^\diamond$ be the corresponding solution to problem \eqref{prob: IQP}.
Add to $\solutions$ the corresponding solution to 
%problem 
\eqref{prob: IQP}, remove \eqref{prob: IQPtilde} from $\problems$, and go back to Step~\ref{step: dec}.

\subsection{Operation count}
\label{sec: count}

In this section we analyze the number of operations performed by our algorithm.

\begin{proposition}
\label{prop: operation count}
The algorithm described in Section \ref{sec: algorithm description} solves at most
\begin{align*}
\left(3 + \left\lceil\sqrt{k \left((2 n \Delta)^2 + \frac{1}{\epsilon}\right)}\right\rceil \right)^k
\end{align*}
ILPs of size polynomial in the size of 
%problem 
\eqref{prob: IQP}. 
Moreover, each ILP has integral data,
at most $n$ variables, 
at most $m$ linear inequalities and possibly additional variable bounds,
and a constraint matrix with subdeterminants bounded by $\Delta$ in absolute value.
\end{proposition}

\begin{proof}
The ILPs solved by our algorithm are problems \eqref{pr: bounds}, \eqref{pr: aux}, \eqref{prob: IQPtilde} when $\tilde k = 0$, \eqref{eq: boundstilde}, and \eqref{prob: ILP on box}.
%First, we show that every ILP solved by our algorithm has the same size of the problem \eqref{prob: IQP}, integral data, and a constraint matrix with subdeterminants bounded in absolute value by $\Delta$.
Any system of inequalities $\widetilde W x \le \tilde w$ in these ILPs is obtained from the original system $Wx \le w$ by fixing 
%and then dropping 
a number of variables $x_i$, $i=1,\dots,k$, to integer values.
Hence, the matrix $\widetilde W$ is a column submatrix of $W$, and the vector $\tilde w$ is integral.
It follows that each ILP has integral data, at most $n$ variables and at most $m$ linear inequalities.
The problems \eqref{pr: aux} and \eqref{prob: ILP on box} have additional variable bounds.
The constraint matrices of these problems are, respectively,
\begin{align*}
W, 
\quad
\begin{pmatrix}
W \\
I \\
-I
\end{pmatrix},
\quad
\widetilde W,
\quad
\widetilde W,
\quad
\begin{pmatrix}
\widetilde W \\
I \\
-I
\end{pmatrix},
\end{align*}
where $I$ denotes the identity matrix.
Therefore, all these constraint matrices have subdeterminants bounded by $\Delta$ in absolute value.
%The constraint matrix of problems \eqref{pr: bounds} is the original $m \times n$ matrix $W$. % with subdeterminants bounded in absolute value by $\Delta$.
%The constraint matrix of problem \eqref{pr: aux} is a $m \times (n-k)$ column submatrix of $W$.
%Finally, the constraint matrix of problems \eqref{prob: IQPtilde}, \eqref{eq: boundstilde}, and \eqref{prob: ILP on box} is $\widetilde W$. %, which is again a column submatrix of $W$.
%This implies that all these constrain matrices have subdeterminants bounded in absolute value by $\Delta$, and have dimension at most $m \times n$.
%\note{These matrices should also have the identity inside.}
To see that each ILP has size polynomial in the size of 
%problem 
\eqref{prob: IQP}, note that 
the vectors $\tilde l, \tilde u$ have size polynomial in the size of $Wx \le w$, and that
$\bar x$ in \eqref{pr: aux} 
can be chosen %such that 
of size polynomial in the size of $Wx \le w$ \cite{SchBookIP}.

In the rest of the proof we show that the algorithm solves in total at most $(3 + g)^k$ ILPs, where $g := \left\lceil\sqrt{k \left((2n\Delta)^2 + 1/\epsilon \right)}\right\rceil$.
We show this statement by induction on the number $k \ge 1$ of variables that appear nonlinearly in the objective.
%For the base case $k=0$ we consider a problem \eqref{prob: IQP} where no variable appears nonlinearly in the objective.
%In this case the algorithm solves the ILP \eqref{pr: aux} in Step~\ref{step: prelim} to detect bundedness and then solves the ILP \eqref{prob: IQP} in the first and only iteration of Step~\ref{step: dec} to find the optimal solution.
%%Note that in this case the ILP \eqref{pr: aux} is the standard linear programming relaxation of the ILP \eqref{prob: IQP}. \note{what?}
%%Since the matrix $B$ is TU and $w$ is integral, 
%Note that the two problems coincide, thus the algorithm only needs to solve one ILP problem.

For the base case $k=1$ we consider 
%a problem 
\eqref{prob: IQP} where one variable appears nonlinearly in the objective.
In Step~\ref{step: prelim} our algorithm solves the two ILPs \eqref{pr: bounds} and the ILP \eqref{pr: aux}.
In Step~\ref{step: dec} the algorithm selects problem \eqref{prob: IQPtilde}$=$\eqref{prob: IQP} from $\problems$, thus we have $\tilde k = 1$.
The algorithm does not need to solve the two ILPs \eqref{eq: boundstilde} since they coincide with problems \eqref{pr: bounds} already solved in Step~\ref{step: prelim}.
Then the algorithm defines $\tilde g := g$.
We first consider the case where we have $\tilde u_1-\tilde l_1 \ge g$.
In this case 
%problem 
\eqref{prob: IQP} does not get decomposed in Step~\ref{step: dec}, and in Step~\ref{step: mesh} the algorithm solves $g$ ILPs.
The total number of ILPs solved is then $3+g$.
Consider now the the remaining case where $\tilde u_1-\tilde l_1 < g$.
In this case 
%problem 
\eqref{prob: IQP} gets decomposed in Step~\ref{step: dec}, and we obtain $\tilde u_1-\tilde l_1+1$ subproblems.
Since $\tilde u_1$, $\tilde l_1$, and $g$ are integers, we have that the number of subproblems is at most $g$.
Each subproblem is a single ILP.
The total number of ILPs is then at most $3 + g$. 

%Note that in this case the ILP \eqref{pr: aux} is the standard linear programming relaxation of the ILP \eqref{prob: IQP}. \note{what?}
%Since the matrix $B$ is TU and $w$ is integral, 
%Note that the two problems coincide, thus the algorithm only needs to solve one ILP problem.
For the induction step, we consider 
%a problem 
\eqref{prob: IQP} with $k \ge 2$ variables that appear nonlinearly in the objective.
In Step~\ref{step: prelim} our algorithm solves the $2k$ ILPs \eqref{pr: bounds} and the ILP \eqref{pr: aux}.
In Step~\ref{step: dec} the algorithm selects problem \eqref{prob: IQPtilde}$=$\eqref{prob: IQP} from $\problems$, thus we have $\tilde k = k$.
The algorithm does not 
%need to 
solve the $2k$ ILPs \eqref{eq: boundstilde} since they coincide with problems \eqref{pr: bounds}.
% already solved in Step~\ref{step: prelim}.
Then the algorithm defines $\tilde g := g$.
We first consider the case where for every index $i \in \{1,\dots, k\}$ we have $\tilde u_i-\tilde l_i \ge g$.
In this case 
%problem 
\eqref{prob: IQP} does not get decomposed in Step~\ref{step: dec}, and in Step~\ref{step: mesh} the algorithm solves $g^k$ ILPs.
The total number of ILPs solved is then $2k+1+g^k \le 3^k + g^k \le (3+g)^k$ since $2k+1 \le 3^k$ for $k \ge 1$.
Consider now the the remaining case where there is an index $i \in \{1,\dots, k\}$ such that $\tilde u_i-\tilde l_i < g$.
In this case 
%problem 
\eqref{prob: IQP} gets decomposed in Step~\ref{step: dec}, and we obtain $\tilde u_i-\tilde l_i+1 \le g$ subproblems.
%Since $\tilde u_1$, $\tilde l_1$, and $g$ are integers, we have that the number of subproblems is at most $g$.
Each subproblem has $n-1$ variables and $k - 1$ variables that appear nonlinearly in the objective. %,  and a constraint matrix with subdeterminants bounded by $\Delta$ in absolute value.
It is simple to see that the number of ILPs that will be solved for each of these subproblems is at most the number of ILPs that would be solved by running the algorithm from scratch with the subproblem as input.
Therefore, by induction, for each subproblem the algorithm solves in total at most 
\begin{align*}
\left(3 + \left\lceil\sqrt{(k - 1) \left((2(n-1) \Delta)^2 + 1/\epsilon \right)}\right\rceil \right)^{k - 1} \le (3 + g)^{k - 1}
\end{align*} 
ILPs.
The total number of ILPs is then at most $2k+1 + g (3+g)^{k - 1}$. 
The latter number is upper bounded by $(3+g)^k$ since $2k+1 \le 3(3+g)^{k-1}$ for every $k \ge 1$.
\end{proof}

\subsection{Correctness of the algorithm}
\label{sec: correctness}

In this section we show that the algorithm detailed in Section \ref{sec: algorithm description} yields an $\epsilon$-approximate solution to 
%problem 
\eqref{prob: IQP}.
Together with Proposition~\ref{prop: operation count}, this provides 
%This completes the 
a proof of Theorem~\ref{th: delta}.

\subsubsection{Feasibility and boundedness}
\label{sec: prelim}

Step~\ref{step: prelim} of the algorithm is analogous to the corresponding part of the algorithm for concave mixed-integer quadratic programming presented in \cite{dP18}.
Moreover, Proposition~1 in \cite{dP18} implies that Step~\ref{step: prelim} of the algorithm correctly determines if 
%problem 
\eqref{prob: IQP} is infeasible 
%and if it is 
or
unbounded.
In particular, if the algorithm continues to Step~\ref{step: dec}, 
%the problem 
then
\eqref{prob: IQP} is feasible and bounded.

\subsubsection{Decomposition}

In this section we show that 
%the decomposition of the problem performed in Step~\ref{step: dec} of the algorithm is correct.
the decomposition of the problem performed in Step~\ref{step: dec} of the algorithm correctly returns an $\epsilon$-approximate solution.

\begin{proposition}
Assume that in Step~\ref{step: mesh} of the algorithm, $x^\diamond$ is an $\epsilon$-approximate solution to the chosen problem \eqref{prob: IQPtilde} in $\problems$.
Then the algorithm correctly returns an $\epsilon$-approximate solution to 
%the input problem 
\eqref{prob: IQP}.
\end{proposition}

\begin{proof}
We have seen in Section~\ref{sec: prelim} that if 
%the input problem 
\eqref{prob: IQP} is infeasible or unbounded, the algorithm correctly detects it in Step~\ref{step: prelim}, thus we now assume that it is feasible and bounded.
In this case, we need to show that the algorithm returns an $\epsilon$-approximate solution to 
%the input problem 
\eqref{prob: IQP}.
To prove this, we only need to show that the algorithm eventually adds to the set $\solutions$ an $\epsilon$-approximate solution $x^\epsilon$ to \eqref{prob: IQP}.
In fact, let $x^\vartriangle$ be the vector returned at the end of the algorithm, i.e., the solution in $\solutions$ with the minimum objective function value when $\problems = \emptyset$.
As the objective value of $x^\vartriangle$ will be at most that of $x^\epsilon$, we have that also the vector $x^\vartriangle$ is an $\epsilon$-approximate solution to \eqref{prob: IQP}.

In this proof it will be useful to lift problems \eqref{prob: IQPtilde} to the space $\R^n$ where 
%the input problem 
\eqref{prob: IQP} lives.
Recall that each 
%problem 
\eqref{prob: IQPtilde} is obtained from 
%problem 
\eqref{prob: IQP} by 
%possibly 
fixing 
%and then dropping 
a number of variables $x_i$, $i=1,\dots,k$, to integer values.
Specifically, say that we fixed $x_i$ to the integer value $\zeta_i$, for every $i \in J$, where $J \subseteq \{1,\dots,k\}$.
The corresponding lifted problem can be obtained from \eqref{prob: IQP} by instead adding the equations $x_i = \zeta_i$, for $i \in J$.
In the remainder of this proof, when we consider a problem \eqref{prob: IQPtilde} we refer to the equivalent lifted version.

We now show that the algorithm eventually adds to $\solutions$ an $\epsilon$-approximate solution $x^\epsilon$ to \eqref{prob: IQP}.
Let $x^*$ be a global optimal solution to \eqref{prob: IQP}.
Let \eqref{prob: IQPtilde} be a problem stored at some point in $\problems$ that 
%contains in the feasible region the vector $x^*$.
contains the vector $x^*$ in the feasible region.
Among all these possible problems, we assume that \eqref{prob: IQPtilde} has a number $\tilde k$ of non-fixed variables that appear nonlinearly in the objective that is minimal.
Note that \eqref{prob: IQPtilde} does not get decomposed in Step~\ref{step: dec}.
Otherwise, the vector $x^*$ would be feasible for one of the subproblems of \eqref{prob: IQPtilde},
% generated in Step~\ref{step: dec}, 
which will have a number of non-fixed variables that appear nonlinearly in the objective that is strictly smaller than $\tilde k$.
Hence, by assumption, when the algorithm selects 
%the problem 
\eqref{prob: IQPtilde} from $\problems$, it adds to $\solutions$ a vector $x^\epsilon$ that is an $\epsilon$-approximate solution to \eqref{prob: IQPtilde}.
Since the feasible region of \eqref{prob: IQPtilde} is contained in the feasible region of \eqref{prob: IQP}, and since the vector $x^*$ is feasible for \eqref{prob: IQPtilde}, it is simple to check that $x^\epsilon$ is an $\epsilon$-approximate solution to 
%the input problem 
\eqref{prob: IQP}.
\end{proof}

\subsubsection{Mesh partition and linear underestimators}

In this section we show that the solution $x^\diamond$ constructed in Step~\ref{step: mesh} is an $\epsilon$-approximate solution to 
%problem 
\eqref{prob: IQPtilde}.

% ===================================

We introduce some definitions in order to simplify the notation.
%Since the matrix $\widetilde W$ is TU and $\tilde w$ is integral, we have that the polyhedron $\tilde \P$ is integral.
We denote by $q : \R^{\tilde n} \to \R$ the nonlinear part of the objective function of \eqref{prob: IQPtilde}, that is,
\begin{align*}
q(x) := \sum_{i=1}^{\tilde k} -\tilde q_i x_i^2.
\end{align*}
Moreover, we denote by $f : \R^{\tilde n} \to \R$ the objective function of \eqref{prob: IQPtilde}, i.e.,
\begin{align*}
f(x) := \sum_{i=1}^{\tilde k} -\tilde q_i x_i^2 + \tilde h^\top x = q(x) + \tilde h^\top x.
\end{align*}
We also define 
\begin{align*}
\gamma := \max \{\tilde q_i (\tilde u_i-\tilde l_i)^2 : i \in \{1,\dots,\tilde k\}\}.
\end{align*}
%Intuitively, an index $i$ that achieves the maximum in the definition of $\gamma$ is a direction along which the function $q$ is the most concave with respect to the feasible region of \eqref{prob: IQPtilde}.
%Intuitively, if $i$ is an index that achieves the maximum in the definition of $\gamma$, then the vector of the standard basis of $\R^{\tilde n}$ along which the function $q$ is the most concave with respect to the feasible region of \eqref{prob: IQPtilde} is precisely $e_i$.
Intuitively, the index $i$ that achieves the maximum in the definition of $\gamma$ indicates the vector $e_i$ of the standard basis of $\R^{\tilde n}$ along which the function $q$ is the most concave with respect to the feasible region of \eqref{prob: IQPtilde}.
As a consequence, the value $\gamma$ provides an indication of how concave is problem \eqref{prob: IQPtilde}.
In order to show that the vector $x^\diamond$ is an $\epsilon$-approximate solution we will derive two bounds: 
(i) an upper bound on the value $f(x^\diamond) - f(x^*)$, where $x^*$ is an optimal solution of \eqref{prob: IQPtilde}, and
(ii) a lower bound on the value $f_{\max} - f(x^*)$, where $f_{\max}$ is the maximum value of $f(x)$ on the feasible region of \eqref{prob: IQPtilde}.
Both bounds will depend linearly on $\gamma$.
We remark that one of the main difficulties in obtaining a polynomial-time algorithm consists in making sure that the dependence on $\gamma$ cancels out in the ratio 
\eqref{eq: epsilon} 
%$
%%\label{eq: epsilon}
%%f(x^\diamond) - f(x^*) \le \epsilon \cdot (f_{\max} - f(x^*)).
%\frac{f(x^\diamond) - f(x^*)}{f_{\max} - f(x^*)}
%$
in the definition of $\epsilon$-approximate solution.
%\begin{align*}
%\frac{f(x^\diamond) - f(x^*)}{f_{\max} - f(x^*)} \le \epsilon.
%\end{align*}

\medskip
\noindent
\ul{An upper bound on the value $f(x^\diamond) - f(x^*)$.} \space
We describe how to obtain an upper bound on the objective function gap between our solution $x^\diamond$ and an optimal solution $x^*$ of \eqref{prob: IQPtilde}.
The derivation of this bound is standard in the context of mesh partition and linear underestimators of separable concave quadratic functions and is based on the lemma that we present next.
%Therefore, we omit the proofs and refer the reader to \cite{Vav92c,dP18}.
The argument is the same used in page~10 in \cite{Vav92c} and in Claim~2 in \cite{dP18}.
We give a complete proof because in these papers the result is not stated in the form that we need.
%The proof techniques that we use in 
%Claim~\ref{claim: sub} and Claim~\ref{claim: good} 
%%the next two claims 
%are standard in the analysis of mesh partition and linear underestimators. % and have been introduced by Vavasis \cite{Vav92c}.
%Claim~\ref{claim: sub}, Claim~\ref{claim: good}, and Claim~\ref{claim: bad} are standard in the analysis of mesh partition and linear underestimators.
%Therefore, the proofs are omitted and we refer to \cite{Vav92c,dP18}.

%For a given box $\C$ constructed in Step~\ref{step: mesh} of the algorithm, the next claim shows that the affine function $\mu: \R^{\tilde k} \to \R$ defined in \eqref{eq: underestimator} is an affine underestimator of the restriction of $q(x)$ to $\C$, and bounds its distance from the quadratic function $q(x)$ in terms of $\gamma, \tilde k$, and $\tilde g$.
%\begin{claim}
%\label{claim: sub}

%\color{blue}

%Consider a box $\C = [r_1,s_1] \times \cdots \times [r_{\tilde k},s_{\tilde k}]$, and the affine function $\eta : \R^{\tilde k} \to \R$ defined in \eqref{eq: underestimator}.
%Then, for every $x \in \C$ we have
%\begin{align*}
%\eta(x) \le q(x) \le \eta(x) + \frac 14 \sum_{i=1}^{\tilde k} \tilde q_i(s_i - r_i)^2.
%\end{align*}
%We will be using the following result regarding linear underestimators of separable concave quadratic functions.
%We refer the reader to \cite{Vav92c,dP18} for its proof.
\begin{lemma}
\label{lem: underestimator}
Let $c : \R^{\tilde k} \to \R$ be a separable concave quadratic function defined by
\begin{align*}
c(x) := \sum_{i=1}^{\tilde k} - c_i x_i^2,
\end{align*}
where $c_i \ge 0$ for $i=1,\dots,{\tilde k}$.
Consider a box $\C = [r_1,s_1] \times \cdots \times [r_{\tilde k},s_{\tilde k}] \subset \R^{\tilde k}$,
and the affine function $\eta : \R^{\tilde k} \to \R$ defined by
%\begin{align}
%\label{eq: underestimator lemma}
\begin{align*}
\eta(x) := \sum_{i=1}^{\tilde k} (-c_i(r_i+s_i) x_i + c_i r_i s_i).
\end{align*}
%\end{align}

Then, for every $x \in \C$, we have
\begin{align*}
\eta(x) \le c(x) \le \eta(x) + \frac 14 \sum_{i=1}^{\tilde k} c_i(s_i - r_i)^2.
\end{align*}
\end{lemma}

%We refer the reader to \cite{Vav92c,dP18} for its proof.
%For each box $\C = [r_1,s_1] \times \cdots \times [r_{\tilde k},s_{\tilde k}]$ constructed in Step~\ref{step: mesh} of the algorithm, the affine function $\eta: \R^{\tilde n} \to \R$ defined in \eqref{eq: underestimator} 
%%is an affine underestimator of the restriction of $q(x)$ to $\C$, and for every $(x_1,\dots,x_{\tilde k}) \in \C$ we have
%satisfies, for every $x \in \R^{\tilde n}$ with $(x_1,\dots,x_{\tilde k}) \in \C$,
%\begin{align}
%\label{eq: lemma}
%\eta(x) \le q(x) \le \eta(x) + \frac 14 \sum_{i=1}^{\tilde k} c_i(s_i - r_i)^2.
%\end{align}

\begin{proof}
%Let $\C$ be a particular box, say $C = [r_1,s_1] \times \cdots \times [r_k,s_k]$, where $s_i - r_i = (\tilde u_i - \tilde l_i)/ \tilde g$ for every $i$.
For each $i = 1,\dots, \tilde k$, we define the affine univariate function
\begin{align*}
%\label{under2}
\eta_i(x_i) := -c_i(r_i+s_i) x_i + c_i r_i s_i.
\end{align*}
The function $\eta_i$
satisfies 
$
\eta_i(r_i) = -c_i r_i^2$, $\eta_i(s_i) = -c_i s_i^2$, and we have
%Note that %the affine function $\eta(x)$ from $\R^k$ to $\R$ defined in \eqref{eq: underestimator} satisfies
\begin{align*}
\eta(x) = \sum_{i=1}^{\tilde k} \eta_i(x_i).
\end{align*}
The separability of $c(x)$ implies that it attains the same values as $\eta(x)$ at all vertices of $\C$.
As the quadratic function $c(x)$ is concave, this in particular implies that $\eta(x) \le c(x)$.

%We now show that $c(x) - \eta(x) \le \gamma \tilde k/(4 \tilde g^2)$.
%From the separability of $q$ and of $\eta$, we obtain 
%\begin{align*}
%c(x) - \eta(x) = \sum_{i=1}^k (-c_i x_i^2 + c_i x_i - \eta_i(x_i)).
%\end{align*}
%Using the explicit formula for $\eta_i$ given in \eqref{under}, it can be derived that
%\begin{align*}
%-c_i x_i^2 + c_i x_i - \eta_i(x_i) = c_i (x_i-r_i) (s_i - x_i).
%\end{align*}
%The univariate quadratic function on the right-hand side is concave, and its maximum is achieved at $x_i = (r_i+s_i)/2$.
%This maximum value is $c_i(s_i - r_i)^2/4 = c_i (\tilde u_i- \tilde l_i)^2/(4\tilde g^2)$.
%As $c_i (\tilde u_i-\tilde l_i)^2 \le \gamma$ for $i=1,\dots,\tilde k$, we establish that $c(x)-\eta(x) \le \gamma \tilde k/(4\tilde g^2)$.

We now show that $c(x) - \eta(x) \le \frac 14 \sum_{i=1}^{\tilde k} c_i(s_i - r_i)^2$.
From the separability of $c$ and of $\eta$, we obtain 
\begin{align*}
c(x) - \eta(x) = \sum_{i=1}^{\tilde k} (-c_i x_i^2 - \eta_i(x_i)).
\end{align*}
Using the definition of $\eta_i$, it can be derived that
\begin{align*}
-c_i x_i^2 - \eta_i(x_i) = c_i (x_i-r_i) (s_i - x_i).
\end{align*}
The univariate quadratic function on the right-hand side is concave, and its maximum is achieved at $x_i = (r_i+s_i)/2$.
This maximum value is $c_i(s_i - r_i)^2/4$, thus we establish that $c(x)-\eta(x) \le \frac 14 \sum_{i=1}^{\tilde k} c_i(s_i - r_i)^2$.
\end{proof}
%\color{black}

%Consider a box $\C = [r_1,s_1] \times \cdots \times [r_{\tilde k},s_{\tilde k}]$, and the affine function $\mu : \R^{\tilde k} \to \R$ defined in \eqref{eq: underestimator}.
%Then, for every $x \in \C$ we have
%\begin{align*}
%\mu(x) \le q(x) \le \mu(x) + \frac 14 \sum_{i=1}^{\tilde k} (\tilde q_i(s_i - r_i)^2).
%\end{align*}
%We will be using the following result regarding linear underestimators of separable concave quadratic functions.
%We refer the reader to \cite{Vav92c,dP18} for its proof.
%\begin{lemma}
%\label{lem: underestimator}
%Let $q : \R^{\tilde k} \to \R$ be a separable concave quadratic function defined by
%\begin{align*}
%q(x) := \sum_{i=1}^{\tilde k} -\tilde q_i x_i^2,
%\end{align*}
%where $\tilde q_i \ge 0$ for $i=1,\dots,{\tilde k}$.
%Consider a box $\C = [r_1,s_1] \times \cdots \times [r_{\tilde k},s_{\tilde k}]$,
%and the affine function $\mu : \R^{\tilde k} \to \R$ defined by
%%\begin{align}
%%\label{eq: underestimator lemma}
%\begin{align*}
%\mu(x) := \sum_{i=1}^{\tilde k} (-\tilde q_i(r_i+s_i) x_i + \tilde q_i r_i s_i).
%\end{align*}
%%\end{align}
%
%
%Then, for every $x \in \C$ we have
%\begin{align*}
%\mu(x) \le q(x) \le \mu(x) + \frac 14 \sum_{i=1}^{\tilde k} (\tilde q_i(s_i - r_i)^2).
%\end{align*}
%\end{lemma}

%We refer the reader to \cite{Vav92c,dP18} for its proof.
Let $\C = [r_1,s_1] \times \cdots \times [r_{\tilde k},s_{\tilde k}]$ be a box constructed in Step~\ref{step: mesh} of the algorithm, and let $\mu: \R^{\tilde n} \to \R$ be the corresponding affine function defined in \eqref{eq: underestimator}.
Lemma~\ref{lem: underestimator} implies that,
%is an affine underestimator of the restriction of $q(x)$ to $\C$, and for every $(x_1,\dots,x_{\tilde k}) \in \C$ we have
%satisfies, 
for every $x \in \R^{\tilde n}$ with $(x_1,\dots,x_{\tilde k}) \in \C$,
\begin{align*}
%\label{eq: lemma}
\mu(x) \le q(x) \le \mu(x) + \frac 14 \sum_{i=1}^{\tilde k} \tilde q_i(s_i - r_i)^2.
\end{align*}
%where 
%$\mu: \R^{\tilde n} \to \R$ 
%$\mu$
%is the affine function defined in \eqref{eq: underestimator}.
\begin{knownproof}
\color{blue}
\begin{proof}
%Let $\C$ be a particular box, say $C = [r_1,s_1] \times \cdots \times [r_k,s_k]$, where $s_i - r_i = (\tilde u_i - \tilde l_i)/ \tilde g$ for every $i$.
For each $i = 1,\dots, \tilde k$, the affine univariate function
\begin{align}
\label{under}
\mu_i(x_i) := -\tilde q_i(r_i+s_i) x_i + \tilde q_i r_i s_i
\end{align}
satisfies 
$
\mu_i(r_i) = -\tilde q_i r_i^2$, and $
\mu_i(s_i) = -\tilde q_i s_i^2$.
Note that the affine function $\mu(x)$ from $\R^k$ to $\R$ defined in \eqref{eq: underestimator} satisfies
\begin{align*}
\mu(x) = \sum_{i=1}^{\tilde k} \mu_i(x_i).
\end{align*}
The separability of $q(x)$ implies that it attains the same values as $\mu(x)$ at all vertices of $\C$.
As the quadratic function $q(x)$ is concave, this in particular implies that $\mu(x) \le q(x)$.

%We now show that $q(x) - \mu(x) \le \gamma \tilde k/(4 \tilde g^2)$.
%From the separability of $q$ and of $\mu$, we obtain 
%\begin{align*}
%q(x) - \mu(x) = \sum_{i=1}^k (-\tilde q_i x_i^2 + c_i x_i - \mu_i(x_i)).
%\end{align*}
%Using the explicit formula for $\mu_i$ given in \eqref{under}, it can be derived that
%\begin{align*}
%-\tilde q_i x_i^2 + c_i x_i - \mu_i(x_i) = \tilde q_i (x_i-r_i) (s_i - x_i).
%\end{align*}
%The univariate quadratic function on the right-hand side is concave, and its maximum is achieved at $x_i = (r_i+s_i)/2$.
%This maximum value is $\tilde q_i(s_i - r_i)^2/4 = \tilde q_i (\tilde u_i- \tilde l_i)^2/(4\tilde g^2)$.
%As $\tilde q_i (\tilde u_i-\tilde l_i)^2 \le \gamma$ for $i=1,\dots,\tilde k$, we establish that $q(x)-\mu(x) \le \gamma \tilde k/(4\tilde g^2)$.

We now show that $q(x) - \mu(x) \le \frac 14 \sum_{i=1}^{\tilde k} (\tilde q_i(s_i - r_i)^2)$.
From the separability of $q$ and of $\mu$, we obtain 
\begin{align*}
q(x) - \mu(x) = \sum_{i=1}^{\tilde k} (-\tilde q_i x_i^2 - \mu_i(x_i)).
\end{align*}
Using the explicit formula for $\mu_i$ given in \eqref{under}, it can be derived that
\begin{align*}
-\tilde q_i x_i^2 - \mu_i(x_i) = \tilde q_i (x_i-r_i) (s_i - x_i).
\end{align*}
The univariate quadratic function on the right-hand side is concave, and its maximum is achieved at $x_i = (r_i+s_i)/2$.
This maximum value is $\tilde q_i(s_i - r_i)^2/4$, thus we establish that $q(x)-\mu(x) \le \frac 14 \sum_{i=1}^{\tilde k} (\tilde q_i(s_i - r_i)^2)$.
\end{proof}
\color{black}
\end{knownproof}
Since $s_i - r_i = (\tilde u_i- \tilde l_i)/\tilde g$ and $\tilde q_i (\tilde u_i-\tilde l_i)^2 \le \gamma$ for $i=1,\dots,\tilde k$, we obtain that, for every $x \in \R^{\tilde n}$ with $(x_1,\dots,x_{\tilde k}) \in \C$,
\begin{align}
\label{eq: claim sub}
\mu(x) \le q(x) \le \mu(x) + \frac{\gamma \tilde k}{4 \tilde g^2}.
\end{align}
%\end{claim}

%Claim~\ref{claim: sub} 
This relation allows us to show the existence of the vector $x^\diamond$.

\begin{claim}
In Step~\ref{step: mesh} the algorithm constructs a feasible solution $x^\diamond$ of \eqref{prob: IQPtilde}.
\end{claim}

\begin{cpf}
We need to show that all the ILPs \eqref{prob: ILP on box} are bounded and that at least one is feasible.

Consider a problem \eqref{prob: ILP on box}.
Note that its feasible region is contained in the feasible region of \eqref{prob: IQPtilde}.
Moreover, in view of \eqref{eq: claim sub}, the objective function of \eqref{prob: ILP on box} is lower bounded by $f(x) - \gamma \tilde k/(4 \tilde g^2)$.
Therefore, the boundedness of \eqref{prob: IQPtilde} established in Step~\ref{step: dec} implies the boundedness of the ILPs \eqref{prob: ILP on box}.

Note that 
%problem 
\eqref{prob: IQPtilde} is feasible, since otherwise the ILPs \eqref{eq: boundstilde} would be infeasible too, and the algorithm would have not entered Step~\ref{step: mesh} with problem \eqref{prob: IQPtilde}.
%By construction the polyhedron $\tilde \P$ is nonempty.
%Since $\tilde \P$ is integral it contains at least an integral vector.
Therefore at least one problem among the ILPs \eqref{prob: ILP on box} is feasible.
This shows that in Step~\ref{step: mesh} the algorithm indeed constructs a feasible solution $x^\diamond$ of \eqref{prob: IQPtilde}.
\end{cpf}

%Let $x^*$ be an optimal solution of \eqref{prob: IQPtilde}.
With a standard argument (see page~11 in \cite{Vav92c} or Claim~3 in \cite{dP18})
%\cite{Vav92c,dP18}, 
we can derive from \eqref{eq: claim sub} that 
%By \eqref{eq: claim sub}, the pointwise difference between $q(x) + d^\top z$ and $\mu(x) + d^\top z$ lies in the interval $[0, \frac{\gamma \tilde k}{4 \tilde g^2}]$, thus so does the difference between their minima:
\begin{align}
\label{eq: claim good}
0 \le f(x^\diamond) - f(x^*) \le \frac{\gamma \tilde k}{4\tilde g^2}.
\end{align}
%\note{Add lemma for above?}

%Next we derive an upper bound on the gap between the objective value at $x^\diamond$ and the objective value at an optimal solution $x^*$ of \eqref{prob: IQPtilde}.
%The bound is a direct consequence of Claim~\ref{claim: sub}. 
%\note{fix this. in case use \eqref{eq: claim sub}}
%\begin{claim}
%\label{claim: good}
%Let $x^\diamond$ be the solution of \eqref{prob: IQPtilde} constructed in Step~\ref{step: mesh} and let $x^*$ be an optimal solution of \eqref{prob: IQPtilde}.
%Then $f(x^\diamond) - f(x^*) \le \gamma \tilde k / (4\tilde g^2)$.
%\end{claim}

\begin{knownproof}
\color{blue}
\begin{proof}
Let $C^\diamond$ be the box that yields the solution $x^\diamond$ and let $\mu^\diamond$ be the corresponding affine function defined in \eqref{eq: underestimator}.
Let $C^*$ be a box such that $x^* \in C^*$ and let $\mu^*$ be the corresponding affine function.
We have
\begin{align*}
f(x^\diamond) & \le \mu^\diamond(x^\diamond) + \tilde h^\top x^\diamond + \frac{\gamma \tilde k}{4\tilde g^2} \\
& \le \mu^*(x^*) + \tilde h^\top x^* + \frac{\gamma \tilde k}{4\tilde g^2} \\
& \le f(x^*) + \frac{\gamma \tilde k}{4\tilde g^2}.
\end{align*}
The first inequality follows because, from \eqref{eq: claim sub}, we have $q(x^\diamond) \le \mu^\diamond (x^\diamond) + \gamma \tilde k/(4\tilde g^2)$.
The second inequality holds by definition of $x^\diamond$.
The third inequality follows because, in view of \eqref{eq: claim sub}, we have $\mu^*(x^*) \le q(x^*)$.
\end{proof}
\color{black}
\end{knownproof}

%\note{Check it works if $f(x^*) = f_{\max}$}

\medskip
\noindent
\ul{A lower bound on the value $f_{\max} - f(x^*)$.} \space
Next, we derive a
%This section is devoted to obtaining a 
lower bound on the gap between the maximum and the minimum objective function values of the
% integral points in the polyhedron 
feasible points of 
%problem 
\eqref{prob: IQPtilde}.
For ease of exposition, we denote by $\tilde \P$ the standard linear relaxation of the feasible region of \eqref{prob: IQPtilde}, i.e.,
\begin{align*}
\tilde \P := \{x \in \R^{\tilde n} : \widetilde Wx \le \tilde w\}.
\end{align*}

While all arguments so far 
%Claim~\ref{claim: sub} and Claim~\ref{claim: good} 
hold even without the assumption that $\tilde u_i-\tilde l_i \ge \tilde g$ for every index $i \in \{1,\dots,\tilde k\}$, this assumption will be crucial to derive this bound.
%\note{We do not actually need this in the next claim but in the one after. Move later?}
%\note{Give example where things do not work without this assumption.}

Without loss of generality, we assume that the index $i \in \{1,\dots,\tilde k\}$ that yields the largest value $\tilde q_i(\tilde u_i - \tilde l_i)^2$ is $i=1$, therefore we have $\gamma = \tilde q_1(\tilde u_1 - \tilde l_1)^2$.
Let $x^l$ be an optimal solution of the ILP defining $\tilde l_1$ in \eqref{eq: boundstilde}.
%Since $\tilde \P$ is integral, the ILP defining $\tilde l_1$ in \eqref{eq: boundstilde} has an integral optimal solution .
Therefore $x^l \in \tilde \P \cap \Z^{\tilde n}$ and $x^l_1 = \tilde l_1$.
Similarly, there is a point $x^u \in \tilde \P \cap \Z^{\tilde n}$ such that $x^u_1 = \tilde u_1$.
We define the midpoint of the segment joining $x^l$ and $x^u$ as
\begin{align}
\label{eq: bullet}
x^\circ := \frac {x^l + x^u}2.
\end{align}
Note that the vector $x^\circ$ is in $\tilde \P$ but is generally not integral.

%Let $x^*$ be an optimal solution of \eqref{prob: IQPtilde}. \note{again?}
Using the properties of vectors $x^l$ and $x^u$ and the assumption on the index $i=1$, the following lower bound on $f(x^\circ) - f(x^*)$ can be derived (see Lemma~6 in \cite{Vav92c} or Claim~4 in \cite{dP18}):
\begin{align}
\label{eq: claim bad}
f(x^\circ) - f(x^*) 
%\ge \frac 14 \sum_{i=1}^{\tilde k} \tilde q_i (y^u_i - y^l_i)^2.
%\ge \frac 14 \tilde q_1 (y_1^u -y_1^l)^2.
\ge \frac \gamma 4.
\end{align}
%\note{add another lemma for the above.}

%\begin{claim}
%\label{claim: bad}
%Let $(y^\circ, z^\circ)$ be the vector defined in \eqref{eq: bullet} and let $x^*$ be an optimal solution of \eqref{prob: IQPtilde}.
%Then 
%$f(x^\circ) - f(x^*) \ge \gamma/4$.
%\end{claim}

\begin{knownproof}
\color{blue}
\begin{proof}
The claim follows from the chain of inequalities below.
\begin{align*}
f(x^\circ) & = \frac 12  \big(f(x^l) + f(x^u)\big) + \frac 14 \sum_{i=1}^{\tilde k} \tilde q_i (x^u_i - x^l_i)^2 \\
& \ge f(x^*) + \frac 14 \sum_{i=1}^{\tilde k} \tilde q_i (x^u_i - x^l_i)^2 \\
& \ge f(x^*) + \frac 14 \tilde q_1 (x_1^u -x_1^l)^2 \\
& = f(x^*) + \frac \gamma 4.
\end{align*}
The first inequality holds because $f(x^l) \ge f(x^*)$ and $f(x^u) \ge f(x^*)$, since $x^l$ and $x^u$ are integral vectors in $\tilde \P$.
In order to obtain the second inequality note that all the terms of the summation are nonnegative, thus a lower bound of the sum is given by the first term.
In the last equation, we have used 
$x_1^l = \tilde l_1$, $x_1^u = \tilde u_1$ by choice of $x^l$ and $x^u$,
and $\gamma = \tilde q_1 (\tilde u_1 - \tilde l_1)^2$ by the assumption on the index $i=1$.
\end{proof}
\color{black}
\end{knownproof}

Define the box 
%\begin{align*}
%D := [\floor{y^\circ_1},\ceil{y^\circ_1}] \times \cdots \times [\floor{y^\circ_k},\ceil{y^\circ_k}].
%\end{align*}
\begin{align*}
\D := [x^\circ_1 - \tilde n \Delta, \ x^\circ_1 + \tilde n \Delta] \times \cdots \times [x^\circ_{\tilde k} - \tilde n \Delta, \ x^\circ_{\tilde k} + \tilde n \Delta] \subset \R^{\tilde k}.
\end{align*}
%Note that the box $\D$ is not necessarily one of the boxes constructed in Step~\ref{step: mesh} of the algorithm.

\begin{claim}
\label{claim: new}
There exist vectors $x^-$, $x^+$ in $\tilde \P \cap \Z^{\tilde n}$ such that $(x_1^-,\dots,x_{\tilde k}^-)$ and $(x_1^+,\dots,x_{\tilde k}^+)$ are in $\D$ and
\begin{align}
\label{eq: bullet prime}
x^\circ = \frac {x^- + x^+}2.
\end{align}
\end{claim}

\begin{cpf}
%To prove this claim we use a classic technique to obtain proximity results for integer problems featuring separable objective functions (see, \eg \cite{GraSko90,HocSha90}).
%The key to obtain our claim is to use this technique in a symmetric way to the pair of vectors $x^\circ, x^l$ and to the pair $x^\circ, x^u$.
%
%Let $x^\circ$ be a feasible solution for problem \eqref{prob: CQP} and let $\hat z$ be a feasible solution for problem \eqref{prob: IQP}.
%We first construct a cone with respect to $x^\circ$ and $z^*$ along the same line as in the proof of Theorem 1 of Granot and Skorin-Kapov [5]. 
%In this proof we use a technique introduced in \cite{GraSko90}.
We first construct the vectors $x^+$ and $x^-$.
To do so, we construct a polyhedral cone 
%similar to the one used in \cite{GraSko90,HocSha90}.
often used to obtain proximity results for integer problems featuring separable objective functions (see, \eg \cite{GraSko90,HocSha90}).
%Let $\{S_1, S_2\}$ be a partition of $\{1, \dots, n\}$ such that for any $i \in S_1$, $x^l_i \ge x^\circ_i$ and for any $i \in S_2$, $x^l_i < x^\circ_i$.
Let $\widetilde W_1$ be the row submatrix of $\widetilde W$ containing the rows $u$ such that $u x^l < u x^\circ$, and let $\widetilde W_2$ be the row submatrix of $\widetilde W$ containing the remaining rows of $\widetilde W$, \ie the rows $u$ such that $u x^l \ge u x^\circ$.
%We partition the matrix $\widetilde W$ into row submatrices $\widetilde W_1$ and $\widetilde W_2$ such that $\widetilde W_1 x^l < \widetilde W_1 x^\circ$ and $\widetilde W_2 x^l \ge \widetilde W_2 x^\circ$, and define the cone 
Consider the polyhedral cone 
\begin{align*}
\T := \{ x \in \R^{\tilde n} : \widetilde W_1 x \le 0, \ \widetilde W_2 x \ge 0\}.
%\ x_i \ge 0, i \in S_1, \ x_i \le 0, i \in S_2, 
%\ x \in \R^n\}.
\end{align*} 
%Note that $x^l - x^\circ \in \T$. 
Let $V \subset \T$ be a finite set of integral vectors that generates $\T$. 
%Then, because of the integrality of the entries of $\widetilde W$, for each $u \in V$, $\norm{u} \le \Delta$, and s
Since $x^l - x^\circ \in \T$, there exist $t \le \tilde n$ vectors $v^1,\dots,v^t$ in $V$, and positive scalars $\alpha_1,\dots,\alpha_t$ such that
% $x^l - x^\circ = \sum_{j=1}^t \alpha_j v^j$. 
%That is,
\begin{align}
\label{eq: HS34}
%x^l = x^\circ + \sum_{j=1}^t \alpha_j v^j.
x^l - x^\circ = \sum_{j=1}^t \alpha_j v^j.
\end{align}
%Let
%$\beta_j := \alpha_j - \floor{\alpha_j}$, for $j=1,\dots,t$,
%and define
We define the vectors $x^+$ and $x^-$ as
\begin{align}
\label{eq: HS35}
x^+ & := x^\circ + \sum_{j=1}^t (\alpha_j - \floor{\alpha_j}) v^j, \\
\label{eq: HS35sym}
x^- & := x^\circ - \sum_{j=1}^t (\alpha_j - \floor{\alpha_j}) v^j.
\end{align}
%and
%\begin{align}
%\label{eq: HS36}
%x^* := x^l - \sum_{j=1}^t (\alpha_j - \floor{\alpha_j}) v^j.
%\end{align}
From \eqref{eq: HS35} and \eqref{eq: HS35sym} we directly obtain \eqref{eq: bullet prime}.

%We show that $(x_1^-,\dots,x_{\tilde k}^-)$ and $(x_1^+,\dots,x_{\tilde k}^+)$ are in $\D$.
Cramer's rule and the integrality of the matrix $\widetilde W$ imply that every vector in $V$ 
%has components bounded by $\Delta$ in absolute value.
can be chosen to have components bounded by $\Delta$ in absolute value.
Since $\alpha_j - \floor{\alpha_j} < 1$, for $j=1,\dots,t$, from \eqref{eq: HS35} and \eqref{eq: HS35sym}, one has, for every $i=1,\dots,\tilde k$,
\begin{align*}
\abs{x_i^+ - x_i^\circ} = \abs{x_i^\circ - x_i^-} = \abs{\sum_{j=1}^t (\alpha_j - \floor{\alpha_j}) v_i^j} \le \tilde n \Delta.
\end{align*}
which implies that $(x_1^-,\dots,x_{\tilde k}^-)$ and $(x_1^+,\dots,x_{\tilde k}^+)$ are in $\D$.

Next, we show that $x^+$ and $x^-$ are in $\Z^{\tilde n}$. % is feasible for problem \eqref{prob: IQP}.
%and that $x^*$ is feasible for problem \eqref{prob: CQP}.
Note that, from \eqref{eq: HS34} and \eqref{eq: HS35} we obtain
\begin{align}
\label{eq: HS37}
x^+ = x^l - \sum_{j=1}^t \floor{\alpha_j} v^j.
\end{align}
%Similarly, from \eqref{eq: HS34} and \eqref{eq: HS36}, one has 
%\begin{align}
%\label{eq: HS38}
%x^* = x^\circ + \sum_{j=1}^t (\alpha_j - (\alpha_j - \floor{\alpha_j})) v^j.
%\end{align}
Since $x^l \in \Z^{\tilde n}$, $\floor{\alpha_j} \in \Z$, and the $v^j$ are integral vectors, we obtain that $x^+$ is integral.
%By \eqref{eq: bullet}, we have $x^\circ - x^u = x^l - x^\circ$.
%Thus, from \eqref{eq: HS34} we have
%\begin{align}
%\label{eq: HS34sym}
%x^\circ - x^u = \sum_{j=1}^t \alpha_j v^j.
%\end{align}
From \eqref{eq: bullet} and \eqref{eq: bullet prime} we have $x^l - x^+ = x^- - x^u$.
%From \eqref{eq: HS34sym} and 
Hence, from \eqref{eq: HS37}, we obtain
\begin{align}
\label{eq: HS37sym}
x^- = x^u + \sum_{j=1}^t \floor{\alpha_j} v^j.
\end{align}
Since also the vector $x^u$ is integral, we obtain $x^- \in \Z^{\tilde n}$.

Next, we show that $x^+$ is in $\tilde \P$.
Using \eqref{eq: HS35}, together with $\alpha_j - \floor{\alpha_j} \ge 0$ and $\widetilde W_1 v^j \le 0$, for $j = 1,\dots,t$, since $v^j \in \T$, we derive
\begin{align*}
\widetilde W_1 x^+ = \widetilde W_1 x^\circ + \sum_{j=1}^t (\alpha_j - \floor{\alpha_j}) \widetilde W_1 v^j \le \widetilde W_1 x^\circ.
\end{align*}
Using \eqref{eq: HS37}, $\floor{\alpha_j} \ge 0$, and $\widetilde W_2 v^j \ge 0$, for $j = 1,\dots,t$, we have
\begin{align*}
\widetilde W_2 x^+ = \widetilde W_2 x^l - \sum_{j=1}^t \floor{\alpha_j} \widetilde W_2 v^j \le \widetilde W_2 x^l.
\end{align*}
Since both vectors $x^\circ$ and $x^l$ satisfy $\widetilde W x \le \tilde w$, we obtain $\widetilde W x^+ \le \tilde w$, thus $x^+ \in \tilde \P$. %\note{direction of inequality is opposite. Fix it.}

% ==================== SYM ========================

%\note{Continue from here. Need to show the things of $x^-$.}

%Next, we show that $x^- \in \tilde \P \cap \Z^{\tilde n}$. % is feasible for problem \eqref{prob: IQP}.
%and that $x^*$ is feasible for problem \eqref{prob: CQP}.
%This proof is symmetric to the proof that $x^+ \in \tilde \P \cap \Z^{\tilde n}$.
%This implies that $x^\circ - x^u \in \T$, and, 
To show $x^- \in \tilde \P$ we use \eqref{eq: HS37sym} and \eqref{eq: HS35sym} to obtain
%with $\alpha_j - \floor{\alpha_j} \ge 0$ and $\widetilde W_2 v^j \ge 0$, for $j = 1,\dots,t$, we obtain
\begin{align*}
\widetilde W_1 x^- & = \widetilde W_1 x^u + \sum_{j=1}^t \floor{\alpha_j} \widetilde W_1 v^j \le \widetilde W_1 x^u, \\
\widetilde W_2 x^- & = \widetilde W_2 x^\circ - \sum_{j=1}^t (\alpha_j - \floor{\alpha_j}) \widetilde W_2 v^j \le \widetilde W_2 x^\circ.
\end{align*}
%Similarly, using \eqref{eq: HS37sym}, $\floor{\alpha_j} \ge 0$, and $\widetilde W_1 v^j \le 0$, for $j = 1,\dots,t$, we derive
%\begin{align*}
%\widetilde W_1 x^- = \widetilde W_1 x^u + \sum_{j=1}^t \floor{\alpha_j} \widetilde W_1 v^j \le \widetilde W_1 x^u.
%\end{align*}
%We have proven that $\widetilde W x^- \le \tilde w$, hence $x^- \in \tilde \P$. %\note{direction of inequality is opposite. Fix it.}
Since vectors $x^u$ and $x^\circ$ satisfy $\widetilde W x \le \tilde w$, we have shown that $x^- \in \tilde \P$. %\note{direction of inequality is opposite. Fix it.}
%\note{IMPORTANT: Can I get $\tilde k \Delta$?}
%Note that, because of the integrality of the matrix $\widetilde W$, we have $\norm{u} \le \Delta$ for each $u \in V$.
%Thus, from \eqref{eq: HS35} and \eqref{eq: HS35sym}, one has
%\begin{align*}
%\norm{x^\circ - v^j} = \norm{x^\circ - x^-} = \norm{\sum_{j=1}^t (\alpha_j - \floor{\alpha_j}) v^j} \le n \Delta,
%\end{align*}
%which implies $y^-, y^+ \in \D$.
%Similarly, from \eqref{eq: HS36}
%\begin{align*}
%\norm{x^* - z^\circ} = \norm{\sum_{j=1}^t (\alpha_j - \floor{\alpha_j}) v^j} \le n \Delta.
%\end{align*}
%From , one has
%\begin{align*}
%\norm{x^\circ - x^-} = \norm{\sum_{j=1}^t (\alpha_j - \floor{\alpha_j}) v^j} \le n \Delta,
%\end{align*}
%which implies $y^+ \in \D$.
\end{cpf}

\begin{claim}
%\label{claim: blackwhite}
\label{claim: yup}
There exists a vector $x^\Yup$ in $\tilde \P \cap \Z^{\tilde n}$ such that 
\begin{align}
\label{eq: yup}
f(x^\circ) - f(x^\Yup) \le \frac{\gamma \tilde k (\tilde n \Delta)^2}{\tilde g^2}.
\end{align}
\end{claim}

%It seems reasonable that both x^+ and x^- are close in obj to x^\circ. Can I prove this? Without using the linear approximation maybe? Maybe gives a better bound too.

\begin{cpf}
%Define the box
%\begin{align*}
%D := [r_1^\Yup,s_1^\Yup] \times \cdots \times [r_{\tilde k}^\Yup,s_{\tilde k}^\Yup].
%\end{align*}
%Note that if $\ceil{r_i^\circ} \le x^\circ_i \le \floor{s_i^\circ}$ for every index $i \in \{1,\dots,\tilde k\}$, then we have $D = C^\circ$. 
%Otherwise the box $D$ is not one of the boxes constructed in Step~\ref{step: mesh} of the algorithm.
%In the next claim we show that we can construct an affine underestimator of the restriction of $q(x)$ to $D$, and we bound its distance from the quadratic function $q(x)$ in terms of $\gamma, \tilde k$, and $\tilde g$.
%While the proof technique is the same that we used in Claim~\ref{claim: sub}, the obtained bound is different due to the different size of the box.
%\note{fix this. in case use \eqref{eq: claim sub}}
%\notemi{$\lambda$ will obviously not depend on the continuous components $y$.}
We define the affine function $\lambda : \R^{\tilde k} \to \R$ that attains the same values as $q(x)$ at the vectors corresponding to the vertices of $\D$:
%of the restriction of $q(x)$ to $\D$
%We define the affine function from $\R^{\tilde k}$ to $\R$ given by
%\note{write that it matches the function on the corners of the box?}
\begin{align*}
\lambda(x) := \sum_{i=1}^{\tilde k} \pare{-2 \tilde q_i x^\circ_i x_i + \tilde q_i ({x^\circ_i}^2 - (\tilde n \Delta)^2)}.
\end{align*}

%The same standard argument that yields \eqref{eq: lemma} implies that
Lemma~\ref{lem: underestimator} implies that
%By Lemma~\ref{lem: underestimator}, 
for every $x \in \R^{\tilde n}$ with $(x_1,\dots,x_{\tilde k}) \in \D$ we have
\begin{align*}
\lambda(x) \le q(x) 
\le \lambda(x) + (\tilde n \Delta)^2 \sum_{i=1}^{\tilde k} \tilde q_i.
\end{align*}
%\notemi{This bound will have sum only over the integer components.}
Using the fact that for each $i = 1,\dots, \tilde k$, we have $1 \le (\tilde u_i- \tilde l_i)^2/\tilde g^2$,
%$\ceil{x^\circ_i} - \floor{x^\circ_i} \le 1 \le (\tilde u_i- \tilde l_i)^2/\tilde g^2$ 
and $\tilde q_i (\tilde u_i-\tilde l_i)^2 \le \gamma$, we derive that,
for every $x \in \R^{\tilde n}$ with $(x_1,\dots,x_{\tilde k}) \in \D$,
\begin{align}
\label{eq: claim sub 2}
\lambda(x) \le q(x)
\le \lambda(x) + \frac{\gamma \tilde k (\tilde n \Delta)^2}{\tilde g^2}.
\end{align}
% TU case:
%\begin{align}
%\label{eq: claim sub 2}
%\lambda(x) \le q(x) 
%\le \lambda(x) + \frac{\gamma \tilde k}{4 \tilde g^2}.
%\end{align}

Consider the linear function $\lambda(x) + \tilde h^\top x$.
As a consequence of Claim~\ref{claim: new}, the vector $x^\circ$ is a convex combination of the vectors $x^-$ and $x^+$.
Hence one of the vectors $x^-$ and $x^+$, that we denote by $x^\Yup$, satisfies 
\begin{align}
\label{eq: claim blackwhite}
\lambda(x^\Yup) + \tilde h^\top x^\Yup \ge \lambda(x^\circ) + \tilde h^\top x^\circ.
\end{align}
In view of Claim~\ref{claim: new}, the vector $x^\Yup$ is in $\tilde \P \cap \Z^{\tilde n}$ and $(x_1^\Yup,\dots,x_{\tilde k}^\Yup) \in \D$.

To complete the proof of the claim, we only need to show that \eqref{eq: yup} holds.
%$
%f(x^\circ) - f(x^\Yup) \le \gamma \tilde k \tilde n \Delta/(2 \tilde g^2).
%$
We have
\begin{align*}
f(x^\circ) & \le \lambda(x^\circ) + \tilde h^\top x^\circ + \frac{\gamma \tilde k (\tilde n \Delta)^2}{\tilde g^2} \\
& \le \lambda(x^\Yup) + \tilde h^\top x^\Yup + \frac{\gamma \tilde k (\tilde n \Delta)^2}{\tilde g^2} \\
& \le f(x^\Yup) + \frac{\gamma \tilde k (\tilde n \Delta)^2}{\tilde g^2}.
\end{align*}
The first inequality follows because, from 
%Claim~\ref{claim: sub 2} 
\eqref{eq: claim sub 2},
we have $q (x^\circ) \le \lambda (x^\circ) + \gamma \tilde k (\tilde n \Delta)^2/ \tilde g^2$.
The second inequality holds as a consequence of \eqref{eq: claim blackwhite}.
In the third inequality we use the fact that $\lambda(x^\Yup) \le q(x^\Yup)$ from 
%Claim~\ref{claim: sub 2}.
\eqref{eq: claim sub 2}.
Hence $f(x^\circ) - f(x^\Yup) \le \gamma \tilde k (\tilde n \Delta)^2 / \tilde g^2$.
\end{cpf}

%\note{No. First, Theorem~\ref{th: gap} is about nonlinear objectives. Second, it will not give us what we want with the current statement. Luckily everything works because $x^\circ$ is a convex combination of two feasible points for the big problem. Here is what we should do. One of those two points has objective $\lambda(x) + d^\top z$ larger than that of $x^\circ$, \ie $\lambda(y^\circ) + d^\top z^\circ$. We denote this point by $x^\square$. We then make Theorem~\ref{th: gap} use the continuous $x^\circ$ and the discrete $x^\square$ to find a vector $x^\Yup \in \D$ that has objective larger then the minimum of the two objectives, \ie larger then $\lambda(y^\circ) + d^\top z^\circ$.
%IMPORTANT: Actually prove the lemma for concave objective. I think that here we do not need to pass by the linear one! We might still need it in the TU case though.}

Combining \eqref{eq: claim bad} with Claim~\ref{claim: yup}, we derive a lower bound on $f(x^\Yup) - f(x^*)$:
\begin{align}
\label{eq: circle}
\begin{split}
f(x^\Yup) - f(x^*) & = \big(f(x^\Yup) - f(x^\circ)\big) + \big(f(x^\circ) - f(x^*)\big) \\
& \ge \frac \gamma 4 - \frac{\gamma \tilde k (\tilde n \Delta)^2}{\tilde g^2}   = \frac{\gamma (\tilde g^2 - \tilde k (2 \tilde n \Delta)^2)}{4 \tilde g^2}.
\end{split}
\end{align}

We are now ready to show that $x^\diamond$ is an $\epsilon$-approximate solution to 
%problem 
\eqref{prob: IQPtilde}.
We have 
\begin{align*}
\frac{f(x^\diamond) - f(x^*)}{f(x^\Yup) - f(x^*)} 
& \le \frac{\cancel{\gamma} \tilde k}{\cancel{4 \tilde g^2}} \cdot \frac{\cancel{4 \tilde g^2}}{\cancel{\gamma} (\tilde g^2 - \tilde k (2 \tilde n \Delta)^2)} 
=  \frac{\tilde k}{\tilde g^2 - \tilde k (2 \tilde n \Delta)^2} 
\le \epsilon.\
\end{align*}
In the first inequality we use 
%Claim~\ref{claim: good} 
\eqref{eq: claim good} 
and \eqref{eq: circle}.
The last inequality holds because by definition of $\tilde g$ we have $\tilde g^2 \ge \tilde k ((2 \tilde n \Delta)^2 + 1/\epsilon)$ which is equivalent to $\tilde k \le \epsilon(\tilde g^2 - \tilde k (2 \tilde n \Delta)^2)$ since $\epsilon > 0$ and to $\tilde k / (\tilde g^2 - \tilde k (2 \tilde n \Delta)^2) \le \epsilon$ since $\tilde g^2 \ge \tilde k (2 \tilde n \Delta)^2$.
This concludes the proof of Theorem~\ref{th: delta}.

\section{The totally unimodular case}

We now consider problem \eqref{prob: IQP} with a TU constraint matrix, thus we fix $\Delta = 1$.
In this section we prove Theorem~\ref{th: TU}.
The proof is very similar to the proof of Theorem~\ref{th: delta}, thus we only describe the differences.

The algorithm is obtained from the one detailed in Section~\ref{sec: algorithm description} by making the following changes: 
(i) The integrality constraint is dropped from all the solved ILPs: \eqref{pr: bounds}, \eqref{pr: aux}, \eqref{prob: IQPtilde} when $\tilde k = 0$, \eqref{eq: boundstilde}, and \eqref{prob: ILP on box};
(ii) The definition of $\tilde g$ in Step~\ref{step: dec}
%, which becomes $\tilde g := \left\lceil\sqrt{\tilde k \left((2 \tilde n)^2 + 1/\epsilon \right)}\right\rceil$ when we fix $\Delta =1$, 
is replaced with 
%the smaller 
$\tilde g := \left\lceil\sqrt{\tilde k \left(1 + 1/\epsilon \right)}\right\rceil$.

The new algorithm solves at most 
\begin{align*}
\left(3 + \left\lceil\sqrt{k \left(1 + \frac{1}{\epsilon}\right)}\right\rceil \right)^k
\end{align*}
LPs of size polynomial in the size of 
%problem 
\eqref{prob: IQP}. 
Moreover, each LP has integral data,
at most $n$ variables, 
at most $m$ linear inequalities and possibly additional variable bounds,
and a TU constraint matrix.
To see this, one just need to go through the proof of Proposition~\ref{prop: operation count} and simply replace the definition of $g$ with $g := \left\lceil\sqrt{k \left(1 + 1/\epsilon \right)}\right\rceil$.

To prove the correctness of the new algorithm, we need to make two modifications to the proof of correctness given in Section \ref{sec: correctness}.
The first modification addresses the change (i) in the description of the algorithm.
The reason why we can drop the integrality constraints is that the original ILPs \eqref{pr: bounds}, \eqref{pr: aux}, \eqref{prob: IQPtilde} when $\tilde k = 0$, \eqref{eq: boundstilde}, and \eqref{prob: ILP on box} all have a TU constraint matrix and integral data, thus they are equivalent to the obtained LPs.

The second modification lies in the derivation of a better lower bound on the value $f_{\max} - f(x^*)$.
To obtain this improved bound we define the box $\D$ differently:
\begin{align*}
\D := [\floor{x^\circ_1},\ceil{x^\circ_1}] \times \cdots \times [\floor{x^\circ_{\tilde k}},\ceil{x^\circ_{\tilde k}}] \subset \R^{\tilde k}.
\end{align*}
With this new definition, Claim~\ref{claim: new} is replaced by the following claim.

\begin{claim}
\label{claim: new TU}
The vector $x^\circ$ lies in the convex hull of the integral vectors in the polyhedron 
\begin{align*}
\P^\Yup := \{ x \in \R^{\tilde n} : \widetilde Wx \le \tilde w, \ \floor{x^\circ_i} \le x_i \le \ceil{x^\circ_i}, i=1,\dots,\tilde k\}.
\end{align*}
\end{claim}

\begin{cpf}
Note that $x^\circ \in \P^\Yup$.
In fact, the vector $x^\circ$ is in $\tilde \P$, thus it satisfies $\widetilde Wx \le \tilde w$, and it trivially satisfies the constraints $\floor{x^\circ_i} \le x_i \le \ceil{x^\circ_i}$ for all $i=1,\dots,\tilde k$.
Moreover, the polyhedron $\P^\Yup$ is integral, since the constraint matrix $\widetilde W$ is TU and $\tilde w$ is integral. 
It follows that the vector $x^\circ$ can be written as a convex combination of integral vectors in $\P^\Yup$. 
\end{cpf}

The next claim takes place of Claim~\ref{claim: yup}.

\begin{claim}
\label{claim: yup TU}
There exists a vector $x^\Yup$ in $\tilde \P \cap \Z^{\tilde n}$ such that 
\begin{align}
\label{eq: yup TU}
f(x^\circ) - f(x^\Yup) \le \frac{\gamma \tilde k}{4 \tilde g^2}.
\end{align}
\end{claim}

\begin{cpf}
We begin by following the same steps performed in the proof of Claim~\ref{claim: yup}, but starting with the new affine function $\lambda : \R^{\tilde k} \to \R$ attaining the same values as $q(x)$ at the vectors corresponding to the vertices of the new box $\D$:
\begin{align*}
\lambda(x) := \sum_{i=1}^{\tilde k} \left( -\tilde q_i(\floor{x^\circ_i} + \ceil{x^\circ_i}) x_i + \tilde q_i \floor{x^\circ_i} \ceil{x^\circ_i} \right).
\end{align*}

With the new definition of $\lambda$, instead of the relation \eqref{eq: claim sub 2}, we derive that,
for every $x \in \R^{\tilde n}$ with $(x_1,\dots,x_{\tilde k}) \in \D$,
\begin{align}
\label{eq: claim sub 2 TU}
\lambda(x) \le q(x)
\le \lambda(x) + \frac{\gamma \tilde k}{4 \tilde g^2}.
\end{align}

\begin{knownproof}
The same standard argument that we used to obtain \eqref{eq: lemma} implies that
%By Lemma~\ref{lem: underestimator}, 
for every $x \in \R^{\tilde n}$ with $(x_1,\dots,x_{\tilde k}) \in \D$ we have
\begin{align*}
\lambda(x) \le q(x) 
\le \lambda(x) + \frac 14 \sum_{i=1}^{\tilde k} \tilde q_i.
\end{align*}
Using the fact that for each $i = 1,\dots, \tilde k$, we have $1 \le (\tilde u_i- \tilde l_i)^2/\tilde g^2$,
%$\ceil{x^\circ_i} - \floor{x^\circ_i} \le 1 \le (\tilde u_i- \tilde l_i)^2/\tilde g^2$ 
and $\tilde q_i (\tilde u_i-\tilde l_i)^2 \le \gamma$, we derive that,
for every $x \in \R^{\tilde n}$ with $(x_1,\dots,x_{\tilde k}) \in \D$,
\begin{align}
\label{eq: claim sub 2 TU}
\lambda(x) \le q(x)
\le \lambda(x) + \frac{\gamma \tilde k}{4 \tilde g^2}.
\end{align}
\end{knownproof}

Consider the linear function $\lambda(x) + \tilde h^\top x$.
In view of Claim~\ref{claim: new TU}, the vector $x^\circ$ lies in the convex hull of the integral vectors in the polyhedron $\P^\Yup$.
Hence there exists an integral vector in $\P^\Yup$, that we denote by $x^\Yup$, which satisfies 
\begin{align}
\label{eq: claim blackwhite TU}
\lambda(x^\Yup) + \tilde h^\top x^\Yup \ge \lambda(x^\circ) + \tilde h^\top x^\circ.
\end{align}
In particular, since $x^\Yup \in \P^\Yup$, we have that $(x_1^\Yup,\dots,x_{\tilde k}^\Yup) \in \D$.

To complete the proof of the claim, one can show that \eqref{eq: yup TU} holds by following the last paragraph of the proof of Claim~\ref{claim: yup TU}, using relations \eqref{eq: claim sub 2 TU} and \eqref{eq: claim blackwhite TU} instead of relations \eqref{eq: claim sub 2} and \eqref{eq: claim blackwhite}.
\begin{knownproof}
To complete the proof of the claim, we only need to show that \eqref{eq: yup TU} holds.
%$
%f(x^\circ) - f(x^\Yup) \le \gamma \tilde k \tilde n \Delta/(2 \tilde g^2).
%$
We have
\begin{align*}
f(x^\circ) & \le \lambda(x^\circ) + \tilde h^\top x^\circ + \frac{\gamma \tilde k}{4 \tilde g^2} \\
& \le \lambda(x^\Yup) + \tilde h^\top x^\Yup + \frac{\gamma \tilde k}{4 \tilde g^2} \\
& \le f(x^\Yup) + \frac{\gamma \tilde k}{4 \tilde g^2}.
\end{align*}
The first inequality follows because from 
%Claim~\ref{claim: sub 2} 
\eqref{eq: claim sub 2 TU}
we have $q (x^\circ) \le \lambda (x^\circ) + \gamma \tilde k / (4\tilde g^2)$.
The second inequality holds as a consequence of \eqref{eq: claim blackwhite TU}.
In the third inequality we use the fact that $\lambda(x^\Yup) \le q(x^\Yup)$ from 
%Claim~\ref{claim: sub 2}.
\eqref{eq: claim sub 2 TU}.
Hence $f(x^\circ) - f(x^\Yup) \le \gamma \tilde k / (4 \tilde g^2)$.
\end{knownproof}
\end{cpf}

Combining \eqref{eq: claim bad} with Claim~\ref{claim: yup TU}, we derive the improved lower bound on $f(x^\Yup) - f(x^*)$:
\begin{align}
\label{eq: circle TU}
\begin{split}
f(x^\Yup) - f(x^*) & = \big(f(x^\Yup) - f(x^\circ)\big) + \big(f(x^\circ) - f(x^*)\big) \\
& \ge \frac \gamma 4 - \frac{\gamma \tilde k}{4 \tilde g^2}   = \frac{\gamma (\tilde g^2 - \tilde k)}{4 \tilde g^2}.
\end{split}
\end{align}

We can now show that $x^\diamond$ is an $\epsilon$-approximate solution to 
%problem 
\eqref{prob: IQPtilde}:
%We have 
\begin{align*}
\frac{f(x^\diamond) - f(x^*)}{f(x^\Yup) - f(x^*)} 
& \le \frac{\cancel{\gamma} \tilde k}{\cancel{4 \tilde g^2}} \cdot \frac{\cancel{4 \tilde g^2}}{\cancel{\gamma} (\tilde g^2 - \tilde k)} 
=  \frac{\tilde k}{\tilde g^2 - \tilde k} 
\le \epsilon.\
\end{align*}
In the first inequality we use 
%Claim~\ref{claim: good} 
\eqref{eq: claim good} 
and \eqref{eq: circle TU}.
The last inequality holds because by definition of $\tilde g$ we have $\tilde g^2 \ge \tilde k (1 + 1/\epsilon)$ which is equivalent to $\tilde k \le \epsilon(\tilde g^2 - \tilde k)$ since $\epsilon > 0$ and to $\tilde k / (\tilde g^2 - \tilde k) \le \epsilon$ since $\tilde g^2 \ge \tilde k$.
This concludes the proof of Theorem~\ref{th: TU}.

\ifthenelse {\boolean{SIOPT}}
{
% For SIOPT - begin
\bibliographystyle{siamplain}
% For SIOPT - end
}
{
% For OO begin
\bibliographystyle{plain}
% For OO end
}

\bibliography{biblio}

\end{document}